\documentclass{article}

\usepackage[a4paper, left=3cm, right=3cm, top=3cm, bottom=3cm, bindingoffset=5mm]{geometry}

\usepackage{mathtools, amssymb, amsthm, upgreek, stackengine}

\usepackage{microtype} 
\usepackage{enumitem, graphicx, caption, float, xcolor, booktabs}

\usepackage{hyperref}
\hypersetup{
    colorlinks=false, 
    linktoc=all,     
    linkcolor=blue,  
}

\newlength\tindent
\numberwithin{equation}{section}

\newtheorem{theorem}{Theorem}[section]
\newtheorem{lemma}[theorem]{Lemma}

\newtheorem{proposition}[theorem]{Proposition}
\newtheorem{assumption}[theorem]{Assumption}
\newtheorem{example}[theorem]{Example}

\newcommand{\R}{\mathbb{R}}
\newcommand{\Z}{\mathbb{Z}}
\newcommand{\N}{\mathbb{N}}

\newcommand{\eps}{\varepsilon}
\newcommand{\defg}{\coloneqq}
\newcommand{\cB}{\mathcal{B}}
\newcommand{\cC}{\mathcal{C}}
\newcommand{\cE}{\mathcal{E}}

\newcommand{\cQ}{\mathcal{Q}}
\newcommand{\cT}{\mathcal{T}}
\newcommand{\cR}{\mathcal{R}}
\newcommand{\transpose}{{\rm T}}

\newcommand{\intervalc}[2]{\textup{[}\hspace{1pt}#1, #2\hspace{1pt}\textup{]}}

\newcommand{\wt}[1]{\widetilde{#1}}
\newcommand{\comment}[1]{}

\definecolor{DarkBlue}{RGB}{14, 13, 117}

\definecolor{DarkRed}{RGB}{130, 7, 7}

\title{Improved Convergence of Multilevel Moving Least-Squares Approximation\thanks{The work has been supported by the DFG
  under  project no. 514588180.}}
\author{Robert Durst\\ Department of Mathematics\\ University of
  Bayreuth\\ 95440 Bayreuth \\ Germany\\robert.durst@uni-bayreuth.de
\and Holger Wendland \\ Department of Mathematics\\ University of
  Bayreuth\\ 95440 Bayreuth \\ Germany\\holger.wendland@uni-bayreuth.de}

\begin{document}

\maketitle
\begin{abstract}
  Moving least-squares approximation is a popular method for
  approximating  multivariate functions from given discrete data. For
  higher accuracy higher 
  degree polynomials have to be used, resulting also in higher
  computational cost and numerical instabilities. Recently, the
  combination of low-order moving least squares with a multilevel
  scheme showed superior numerical behavior. In this paper we will
  prove, amongst other things,  that such a combination of  moving
  least-squares with a multilevel scheme indeed leads to improved
  convergence results, at least if the data sites form a regular
  grid.  
  \end{abstract}


\section{Introduction}

The approximation or learning of an unknown function from given, discrete
data is one of the core problems in approximation and learning
theory. Aside from neural networks, the most popular and
mathematically best understood methods comprise
kernel-based learning (\cite{Schoelkopf-Smola-02-1,
  Cristianini-Shawe-Taylor-00-1, Steinwart-Christmann-08-1}), radial
basis function approximation (\cite{Fasshauer-07-1,Fornberg-Flyer-15-1,
  Wendland-05-1}) and moving least-squares (MLS) approximation 
(\cite{Farwig-86-1,Lancaster-Salkauskas-81-1,Levin-98-1,Wendland-01-1}). The
latter method particularly has gained popularity in computer-aided design
applications (\cite{Alexa-etal-01-1,Alexa-etal-03-1}). Theoretically,
the approximation properties of MLS approximation, as well as its drawbacks, are well
established, see for example \cite{Levin-98-1,
  Melenk-05-1,Mirzaei-15-1,Wendland-01-1}.

More precisely, let $d \in\N$ be an integer denoting the spatial dimension. We
consider the problem of approximating an unknown function $f: \R^d \to
\R$ from its values on a discrete set of sample sites $X = \{x_j : j
\in J\}$. These samples $(x_j, f(x_j))$ typically represent
measurements of physical quantities, biological observations, or
outputs from numerical simulations.  MLS approximation constructs an 
approximant $Q_X f$ that reproduces polynomials of degree $r \in
\N_0$; that is, for all $p \in \pi_r(\R^d)$,  
\[
Q_X p(x) = p(x), \quad x \in \R^d.
\]
In theory, the accuracy of MLS can be improved by increasing the
reproduction degree $r$. In practice, however, $r$ is usually
restricted to $\{0, 1, 2\}$ primarily for two reasons. First, every
evaluation of the MLS approximant requires solving a dense, symmetric
linear system of size $M \times M$, where $M = \dim(\pi_r(\R^d)) =
\binom{d+r}{d}$. Consequently, high reproduction degrees lead to
computational overhead. Second, these systems become increasingly
ill-conditioned as $r$ increases.  

It is therefore desirable to enhance the accuracy of MLS schemes with
low reproduction degrees without sacrificing numerical stability. To
achieve this, we analyze a multiscale scheme based on the multilevel
error-correction method introduced in \cite{Sharon-etal-23-1} (cf. Algorithm
4.1 therein), which has previously successfully been used in the context of
radial basis function interpolation and approximation, see for example
\cite{Floater-Iske-96-1, LeGia-etal-10-1, Wendland-10-1,
  Wendland-17-1}. The authors of \cite{Sharon-etal-23-1} demonstrated
that, under some additional assumptions, this multilevel 
method converges at a rate at least equal to the classical MLS
approach. However, their numerical experiments suggested that
significantly higher convergence rates are achievable. This work
substantiates that conjecture by obtaining improved convergence rates
for cases where the sample sites are structured as uniform grids. 

Furthermore, this work extends convergence results for standard MLS
and, in a certain way, the results in
\cite{Franz-Wendland-22-1}, which analyzed a similar grid-based
multilevel approach for kernel-based 
quasi-interpolation without polynomial reproduction. In that context,
it was shown that a non-converging quasi-interpolation method could be
transformed into a convergent scheme through multilevel error
correction.

The main results of this paper can be summarized as follows.
\begin{enumerate}
\item In the case of an even polynomial reproduction degree $r$, the
  expected convergence order of $r+1$ can be raised to $r+2$. To be
  more precise, we derive an approximate approximation theory result
  in the spirit of \cite{Mazya-Schmidt-07-1} but in a different
  setting and  using completely
  different techniques, see Theorem \ref{thm:classic_error_improved}.
\item The multilevel MLS approximation method significantly improves
  the simultaneous approximation results of standard MLS, see Theorem
  \ref{thm:result} and the discussion directly afterwards. On the one hand
  we are able to derive faster convergence for approximating the function
  and all its derivatives up to order $r$, reaching almost twice the
  order of standard MLS, with an additional speedup being obtained in
  the pre-asymptotic regime. On the other hand, we can extend the
  convergence result to derivatives of order $r+1$ instead of only $r$.
\end{enumerate}

This paper is organized as follows. In the next section, we will
recall the classical MLS framework and show 
improved results tailored to our specific situation. In
Section~\ref{sec:MLS} we will discuss the multilevel 
version of the MLS approximation and prove our main result on
its convergence. The final section is 
devoted to further numerical evidence. 

We end this section by stating our primary assumptions. In the
following, the standard Euclidean norm on $\R^d$ is denoted by
$\|\cdot\|$. 

\begin{assumption}\label{general_assumption}
Throughout this work, $d, k, m \in \N$ and $r \in \N_0$ are constants that
are used in the following way. As usual, $d$ denotes the space
dimension. The target functions $f$ belongs to $C^k(\R^d)$, and the
MLS approximation procedure is determined by a non-negative, non-zero
kernel $\Phi \in  C^{m}(\R^d)$ with compact support in the unit ball,
reproducing polynomials of degree at most $r$. The space of these
polynomials is denoted by $\pi_r(\R^d)$ and equipped with a basis
$\cB:=\{p_1,\dots,p_M\}$, where $M=\dim(\pi_r(\R^d))$. We assume that
the parameters satisfy  
\[
m > k > r.
\]
\end{assumption}

The constants occurring in our analysis will frequently depend on the
tuple $\cT = (\Phi,d,m,r,\cB)$ as well as the smoothness parameter
$k$, which we consider fixed throughout this paper unless stated
otherwise. 


\section{Classical MLS}\label{sec:classic_mls}
In this section we collect and prove necessary results on the
classical MLS method. We mainly follow standard procedure. However,
classically, the  MLS method uses a finite number of data, where the
data sites are scattered in a bounded domain $\Omega\subseteq\R^d$. In
this paper, we will assume that the data sites form a regular infinite
grid. This leads to subtle differences, particular when it comes to
studying the approximation properties over a bounded domain, as we
will now also use points outside the domain.  

For each point $x \in \R^d$, the MLS approximant is evaluated by
solving a linear equation system involving the so-called shape matrix
$A(x)$. In Subsection~\ref{sec:A_x}, we derive estimates on the
spectrum of $A(x)^{-1}$ and its element-wise derivatives with respect
to $x$. These results will benefit from the regular structure of the
data points and will be the driving force behind the convergence
speedup of both the classical and the multilevel method. 

In Subsection~\ref{sec:superconvergence}, we give a convergence proof for
the classical multilevel method. Afterwards, we show that due to the
special structure of our point set, an additional convergence order
can be gained if the polynomial reproduction degree $r$ is even. 

We will now commence by introducing the classical
MLS method. Let $h>0$ and let 
$f:\R^d\to\R$ be a function that is only known on the grid 
$ X_h := h\Z^d$. Then,  for $x\in\R^d$, the MLS approximant $Q_h
f:\R^d\to\R$ is defined as 
$ Q_h f(x) := p^*(x)$,  where $p^*=p_x^*$ is the solution of the minimization problem
\begin{equation}\label{mls_problem}
\min\left\{
\sum_{q\in\Z^d}\Phi\left(\frac{x-hq}{\delta}\right)\left|f(hq)-p(hq)\right|^2
: p\in\pi_r(\R^d) \right\},
\end{equation}
provided that a unique solution exists for all $x\in\R^d$. Here,
$\delta := \nu h>0$ is called the \emph{support radius}, since the
function  $\Phi(\cdot/\delta)$ is supported in a ball of radius $\delta$. The parameter
$\nu>1$ describes the ratio between the support radius and the mesh size.
The corresponding error operator is defined by
\[
E_h f := f - Q_h f.
\]

We now discuss conditions under which the MLS approximant $Q_h f$ is
well defined. Instead of using the basis $\cB = \{p_1,\ldots,p_M\}$ from
Assumption~\ref{general_assumption} directly, we will use the
$x$-dependent basis  
\[
\left\{p_j\left(\frac{x-\cdot}{\delta}\right) : 1\le j\le M\right\}
\]
in the  minimization problem (\ref{mls_problem}).
Expressing an arbitrary polynomial $p\in\pi_r(\R^d)$ in this basis as
$p=p_v = \sum_{j=1}^M v_j p_j((x-\cdot)/\delta)$ allows us to rewrite
the term in (\ref{mls_problem}) to be minimized as
\[
  \sum_{q\in\Z^d}\Phi\left(\frac{x-hq}{\delta}\right)\left|f(hq)-\sum_{j=1}^M
  v_j p_j\left(\frac{x-hq}{\delta}\right)\right|^2  = v^\transpose A(x)v - 2v^\transpose b(x) + c(x),
\]
where $b(x) := (b_1(x),\dots,b_M(x))^\transpose$ and $c(x)$ are given by
\[
b_j(x) := \sum_{q\in\Z^d}\Phi\biggl(\frac{x-hq}{\delta}\biggr)p_j\biggl(\frac{x-hq}{\delta}\biggr)f(hq),
\qquad
c(x) := \sum_{q\in\Z^d} \Phi\biggl(\frac{x-hq}{\delta}\biggr) |f(hq)|^2
\]
and the matrix $A(x)\in\R^{M\times M}$ has entries
\begin{equation}\label{matrixA}
a_{ij}(x) = \sum_{q\in\Z^d}
\Phi\left(\frac{x-hq}{\delta}\right)p_i\left(\frac{x-hq}{\delta}\right)p_j
\left(\frac{x-hq}{\delta}\right).
\end{equation}
The matrix $A(x)$ is obviously symmetric and positive
semi-definite. If $A(x)$ is even positive definite then we have a
unique solution $p_{v^*}$, where $v^*$ is given by $v^*=A(x)^{-1}
b(x)$ and the MLS approximant takes the form
\[
Q_hf(x) = p_{v^*}(x) = \sum_{j=1}^M v_j^* p_j(0) = \omega_0^\transpose
v^* = \omega_0^\transpose A(x)^{-1} b(x)
\]
with 
$ \omega_0 := \left(p_1(0),\ldots,p_M(0)\right)^\transpose \in \R^M$.
This essentially proves the following proposition.

\begin{proposition}
Suppose that $A(x)$ is positive definite for all $x\in\R^d$, and define
\[
\lambda(x) := A(x)^{-1}\omega_0 = A(x)^{-1}\left(p_1(0),\ldots p_M(0)\right)^{\transpose}.
\]
Then the minimization problem~\eqref{mls_problem} admits a unique solution and $Q_h$ is a linear operator which can be written in the form 
\[
Q_h f(x) = \sum_{q\in\Z^d} w_\delta(hq,x)\,f(hq),
\]
where
\begin{equation}\label{def:weightfct}
w_\delta(y,x)
:=
\Phi\biggl(\frac{x-y}{\delta}\biggr)\sum_{j=1}^M \lambda_j(x)\,p_j\biggl(\frac{x-y}{\delta}\biggr),
\qquad x,y\in\R^d.
\end{equation}
The function $w_\delta$ does not depend on the choice of the basis $\cB$. Moreover, $Q_h$ reproduces polynomials of degree $r$, i.e. 
$Q_hp = p$ for all $p\in\pi_r(\R^d)$.
\end{proposition}
\begin{proof}
  The discussion before this proposition immediately shows
  \begin{align*}
  Q_hf(x) &= (A(x)^{-1} \omega_0)^\transpose b(x) =
  \lambda(x)^\transpose b(x) =\sum_{j=1}^M \lambda_j(x)
  \sum_{q\in\Z^d} \Phi\left(\frac{x-hq}{\delta}\right) p_j
  \left(\frac{x-hq}{\delta}\right) f(hq)\\
  & = \sum_{q\in\Z^d} w_\delta(hq,x) f(hq),
  \end{align*}
  which is the stated representation with the given weight functions
  from (\ref{def:weightfct}), showing also that $Q_h:C(\R^d)\to
  C(\R^d)$ is a linear map. The independence of $w_\delta$ from $\cB$ follows immediately from the representation
\[
w_\delta(y,x) = \Phi\biggl(\frac{x-y}{\delta}\biggr) \cdot (p_1(0),\ldots,p_M(0)) A(x)^{-1} \biggl( p_1\biggl(\frac{x-y}{\delta}\biggr),\ldots,p_M\biggl(\frac{x-y}{\delta}\biggr)\biggr)^\transpose
\]
and the definition of $A(x)$. It remains to prove the polynomial reproduction property.
Let $x\in\R^d$ and $p\in\pi_r(\R^d)$. Then there exists a vector $z(x)\in\R^M$
such that $ p = \sum_{i=1}^M z_i(x)\,p_i((x-\cdot)/{\delta})$. The
just proven representation of $Q_h$ and its linearity then yield
\begin{align*}
Q_h p(x)
&= \sum_{i=1}^M z_i(x)\, Q_h\!\left(p_i\biggl(\frac{x-\cdot}{\delta}\biggr)\right)(x) \\
&= \sum_{i=1}^M \sum_{j=1}^M z_i(x)\lambda_j(x) \sum_{q\in\Z^d}
\Phi\biggl(\frac{x-hq}{\delta}\biggr)
p_j\biggl(\frac{x-hq}{\delta}\biggr)
p_i\biggl(\frac{x-hq}{\delta}\biggr) \\
&= \sum_{i=1}^M \sum_{j=1}^M z_i(x)\lambda_j(x)a_{i j}(x) = z(x)^\transpose A(x)\lambda(x) \\
&= z(x)^\transpose A(x)A(x)^{-1}\omega_0 = z(x)^\transpose \omega_0 = p(x).
\end{align*}
This completes the proof.
\end{proof}

The function $w_\delta$ defined in~\eqref{def:weightfct} is referred to as the \emph{weight function}, and the matrix $A(x)$ as the \emph{shape matrix}.
For $q \in \Z^d$, we will occasionally use the shorthand notation
\[
w_q : \R^d \to \R, \qquad w_q(x) := w_\delta(hq,x),
\]
to simplify expressions. It follows immediately from their construction that the functions $w_q$ belong to $C^{m}(\R^d)$. Moreover, the weight functions satisfy the following family of identities, which will be relevant later on.

\begin{lemma}\label{lem:weightfct_mixed_derivatives}
Suppose that $A(x)$ is positive definite for every $x \in \R^d$. Then, for any
pair of multi-indices $\alpha,\beta \in \N_0^d$ with
$\alpha \neq \beta$, $|\alpha|\leq r$, and $|\beta|\leq m$,
it holds that
\begin{equation}\label{claim:weightfct_mixed_derivatives}
\sum_{q\in\Z^d} D^\beta w_q(x)\,(x-hq)^\alpha = 0
\end{equation}
for all $x \in \R^d$.
\end{lemma}

\begin{proof}
We prove~\eqref{claim:weightfct_mixed_derivatives} by induction on
$|\beta |$ and start with the base case. If $|\beta |=0$ and
$\alpha\neq\beta$, the claim follows immediately from the polynomial
reproduction property $Q_hp=p$ for all $p\in\pi_r(\R^d)$. For the
induction step, assume that \eqref{claim:weightfct_mixed_derivatives}
holds for all multi-indices $\beta$ with 
$0 \leq |\beta |\leq s < m$, and let $|\beta |= s+1$. Then
\begin{align*}
0
&= D^\beta \biggl[ \sum_{q\in\Z^d} w_q(x)\,(x-hq)^\alpha \biggr] \\
&= \sum_{q\in\Z^d} D^\beta w_q(x)\,(x-hq)^\alpha
 + \sum_{0<\gamma\leq\beta} \binom{\beta}{\gamma}
   \sum_{q\in\Z^d} D^{\beta-\gamma} w_q(x)\,
   D^\gamma\!\bigl[(x-hq)^\alpha\big] \\
&= \sum_{q\in\Z^d} D^\beta w_q(x)\,(x-hq)^\alpha
 + \sum_{\substack{0<\gamma\leq\beta \\ \gamma\leq\alpha}}
   \binom{\beta}{\gamma} (\alpha-\gamma)!
   \sum_{q\in\Z^d} D^{\beta-\gamma} w_q(x)\,(x-hq)^{\alpha-\gamma}.
\end{align*}
Since $\alpha \neq \beta$, it follows that $\alpha-\gamma \neq \beta-\gamma$ for
all $0<\gamma\leq\alpha,\beta$. Therefore, the induction hypothesis implies
\[
\sum_{q\in\Z^d} D^{\beta-\gamma} w_q(x)\,(x-hq)^{\alpha-\gamma} = 0,
\]
which completes the proof.
\end{proof}

\subsection{Spectral Estimates}\label{sec:A_x}

As announced, we are going to investigate the properties of the shape
matrix $A(x)$ and its inverse. In particular, we focus on two 
aspects. First, we show that $A(x)$ is positive definite for
sufficiently large values of $\nu$, where $\nu=\delta/h$ is the
ratio between the support radius $\delta$ and the grid width $h$.
Second, we study the decay, as 
$\nu \to \infty$, of the spectral norms of the matrices 
\[
D^\gamma\bigl(A^{-1}(x)\bigr), \qquad \gamma \in \N_0^d,\quad 0 \leq |\gamma|\leq m,
\]
and estimate the corresponding decay rates. Here, $D^\gamma A(x)$ and
$D^\gamma\bigl(A^{-1}(x)\bigr)$ will denote the component-wise
derivatives with respect to $x$. Our approach is based on a classical
quadrature result going back to P.A. Raviart \cite{Raviart-85-1}. In
the following, we denote the $h$-cube centered at $hq$ with $q \in \Z^d$ by 
\[
V_{h,q} \defg \{hq\} + {\intervalc {-h\slash 2} {h\slash 2}}^d.
\]

\begin{proposition}\label{prop:rav_general}
Suppose $\ell \in \mathbb{N}_0$. Then there is a constant $C =
C(d,\ell)>0$ such that for all continuous functions $g \in
C_c^\ell(\R^d)$  with compact support we have 
\begin{equation}\label{claim:rav_general}
\biggl|\sum_{q\in\mathbb{Z}^d} g(hq) - \frac{1}{h^d}\int_{\R^d}
g(y)\,dy \biggr|\leq C h^{\ell}
\sum_{q\in\mathbb{Z}^d}|g|_{W^{\ell,\infty}(V_{h,q})}. 
\end{equation}
\end{proposition}

Note that the results in \cite{Raviart-85-1}, in particular Lemma 3.2,
Lemma 3.3 and Theorem 3.1, all hold in the case of  $p = \infty$ and $q = 1$,
even if this fact is not explicitly mentioned there. The proofs for
this case are essentially the same. However, we do not have an
explicit bound for the constant $C$, as it stems from an abstract
argument involving the Bramble-Hilbert lemma. To shorten some of our
arguments, we formulate a specialized version of
Proposition~\ref{prop:rav_general} tailored to our needs. 

\begin{proposition}\label{prop:rav}
Suppose $\ell \in \mathbb{N}_0$. Then there is a constant $C^* =
C^*(d,\ell)>0$ such that for all functions $g \in C_c^\ell(\R^d)$ which
are compactly supported in a ball of radius $\delta = \nu h$, we have 
\begin{equation}\label{claim:rav}
\biggl| \sum_{q\in\mathbb{Z}^d} g(hq) - \frac{1}{h^d}\int_{\R^d}
g(y)\,dy \biggr|\leq C^* h^{\ell} \nu^d |g|_{W^{\ell,\infty}(\R^d)}. 
\end{equation}
\begin{proof}
This is an obvious consequence of Proposition~\ref{prop:rav_general}
and the fact that the number of cubes $V_{h,q}$ which intersect any
ball $B_\delta(x)$, $x \in \R^d$, is bounded by a constant
proportional to $\nu^d$. 
\end{proof}
\end{proposition}

 We will now apply this classic result to estimate the spectral norms
 of $A(x)$ and its derivatives. As announced, the following lemma also
 shows that $A(x)$ is positive definite for all $x \in \mathbb{R}^d$,
 provided that $\nu$ is chosen sufficiently large. We denote the
 spectral norm by $\Vert\cdot\Vert_2$. 

\begin{lemma}\label{lem:Ax_spec_est}
Let $\gamma \in \mathbb{N}_0^d$ with $|\gamma|\leq m$. Then there
exist constants $C_\gamma = C_\gamma(\cT)>0$ and $\tau = \tau(\cT) > 0$
such that for all $x \in \mathbb{R}^d$, 
\begin{equation}
\label{claim:Ax_spec_est_1}
\Vert D^\gamma A(x) \Vert_2 \leq C_\gamma\, \delta^{-|\gamma|}\,
\nu^{d+|\gamma|-m},\qquad \gamma\not= 0, 
\end{equation}
and
\begin{equation}
\label{claim:Ax_spec_est_2}
\lambda_{\mathrm{min}}(A(x)) \geq \nu^d\big(\tau - C_0 \nu^{-m}\bigr).
\end{equation}
\begin{proof}
Fix $x \in \R^d$. To simplify our notation, we introduce the two auxiliary functions
\[
\Phi_\alpha(y) := \Phi(y)\,y^\alpha \quad \text{and} \quad \Psi_\alpha(y) \defg \Phi\left(\frac{x-y}{\delta}\right)\left(\frac{x-y}{\delta}\right)^\alpha
\]
for $y \in \R^d$ and $\alpha \in \mathbb{N}_0^d$. Furthermore, for $v
\in \R^M$ with $\Vert v\Vert = 1$ we define $p_v \coloneqq
\sum_{j=1}^M v_j p_j$ to be the polynomial represented by the
coefficients of $v$ in the basis $\cB$ from
Assumption~\ref{general_assumption}. Since $p_v^2$ is a polynomial of
degree at most $2r$, there exist polynomials 
$c_\alpha \in \pi_2(\R^M)$ such that
\[
p_v^2(y) = \sum_{|\alpha|\leq 2r} c_\alpha(v)\, y^\alpha.
\]
Using the definition of $A(x)$ from (\ref{matrixA}), for every multi-index $|\gamma|\leq m$ we obtain
\[
v^\top D^\gamma A(x) v = (-1)^{|\gamma|} \sum_{|\alpha|\leq 2r}
c_\alpha(v) \sum_{q \in \mathbb{Z}^d} D^\gamma \Psi_\alpha(hq). 
\]
Since $\Phi$ is compactly supported, the integral of $D^\gamma
\Psi_\alpha$ over $\R^d$ vanishes in the case of $\gamma \not=
0$. Furthermore, $D^\alpha\Psi_\alpha$ satisfies the prerequisites of
Proposition~\ref{prop:rav} with $\ell = m -|\gamma|$, yielding 
\begin{align*}
\biggl|\sum_{q \in \Z^d} D^\gamma \Psi_\alpha(hq) \biggr|
\leq  
C^* \nu^{d} h^{m-|\gamma|} |\Psi_\alpha |_{W^{m,\infty}(\R^d)} = C^*
\nu^{d-m} h^{-|\gamma|} |\Phi_\alpha |_{W^{m,\infty}(\R^d)}. 
\end{align*}
This shows that for every $\gamma \not= 0$ there exits a constant
$C_{\gamma}' = C_{\gamma}'(\cT)>0$ such that 
\[
|v^\top D^\gamma A(x) v|\leq C_\gamma'\, \nu^{d-m} h^{-|\gamma|}\sum_{|\alpha|\leq 2r} |c_\alpha(v) |.
\]
The sum appearing on the right hand side is continuous in $v$ and
therefore bounded on the unit sphere,
proving~\eqref{claim:Ax_spec_est_1}. In the case of $\gamma=0$, we use
the triangle inequality and the fact that $A(x)$ is positive
semi-definite to see that 
\begin{align*}
&v^T A(x) v = |v^T A(x) v|= 
\\ &\quad= \biggl|\frac{1}{h^d}\int_{\R^d}
\Phi\left(\frac{x-y}{\delta}\right)p_v^2\left(\frac{x-y}{\delta}\right)\,dy
+ \sum_{|\alpha|\leq 2r} c_\alpha(v) \biggl(\sum_{q\in\Z^d}
\Psi_\alpha(hq) - \frac{1}{h^d}\int_{\R^d}
\Psi_\alpha(y)\,dy\biggr)\biggr| 
\\ &\quad\geq \biggl|\frac{1}{h^d}\int_{\R^d}
\Phi\left(\frac{x-y}{\delta}\right)p_v^2\left(\frac{x-y}{\delta}\right)\,dy
\biggr|- \biggl|\sum_{|\alpha|\leq 2r} c_\alpha(v)
\biggl(\sum_{q\in\Z^d} \Psi_\alpha(hq) - \frac{1}{h^d}\int_{\R^d}
\Psi_\alpha(y)\,dy\biggr) \biggr|\\ &\quad \geq \nu^d \min_{\Vert
  v\Vert = 1} \Vert \Phi p_v^2 \Vert_{L^1(\R^d)} -
\biggl(\sum_{|\alpha|\leq 2r} |c_\alpha(v)|\biggr) \max_{|\alpha|\leq
  2r} \biggl|\sum_{q\in\Z^d} \Psi_\alpha(hq) -
\frac{1}{h^d}\int_{\R^d} \Psi_\alpha(y)\,dy \biggr| 
\\ &\quad\geq \nu^d \min_{\Vert v\Vert = 1} \Vert \Phi p_v^2
\Vert_{L^1(\R^d)} - C^* \nu^{d-m} \biggl(\sum_{|\alpha|\leq 2r}
|c_\alpha(v)|\biggr) \max_{|\alpha|\leq 2r}
|\Phi_\alpha|_{W^{m,\infty}(\R^d)}, 
\end{align*}
where we used Proposition~\ref{prop:rav} once more in the last
step. The sum over all $|c_\alpha |$ is continuous, and so is the
function $v \mapsto \Vert \Phi p_v^2 \Vert_{L^1(\R^d)}$. Hence both
maps are bounded from above and below on the unit sphere, the latter
attaining a positive minimum. This shows~\eqref{claim:Ax_spec_est_2}. 
\end{proof}
\end{lemma}

In particular, the previous lemma shows that there exists $\nu_* = \nu_*(\cT) > 0$ such that
for all $\nu \geq \nu_*$ and all $x \in \R^d$,
\begin{equation}
\label{Ax_smallest_eval}
\lambda_{\mathrm{min}}(A(x))
\geq \tau_*\, \nu^d,
\qquad
\tau_* := \frac{\tau}{2}.
\end{equation}
Hence, $A(x)$ is positive definite for all sufficiently large $\nu$.
We now turn to estimates for the norm of the inverse matrix and its derivatives.

\begin{lemma}\label{lem:Ax_inv_spec_est}
Suppose $\nu \geq \nu_*$. Then
\begin{align}\label{claim:Ax_inv_spec_est_1}
\Vert A(x)^{-1} \Vert_2 \leq \tau_*^{-1}\,\nu^{-d}.
\end{align}
Also, for each $\gamma \in \mathbb{N}_0^d$ with $0 < |\gamma|\leq m$ there exists a constant
$\wt{C}_{\gamma} = \wt{C}_{\gamma}(\cT) > 0$
such that
\begin{align}\label{claim:Ax_inv_spec_est_2}
\Vert D^\gamma (A(x)^{-1}) \Vert_2
\leq \wt{C}_{\gamma}\,\delta^{-|\gamma|}\,\nu^{|\gamma|-m-d}
\end{align}
for all $x \in \R^d$.
\begin{proof}
Estimate~\eqref{claim:Ax_inv_spec_est_1} follows immediately from
\[
\Vert A(x)^{-1} \Vert_2 = \frac{1}{\lambda_{\textup{min}}(A(x))}
\]
together with~\eqref{Ax_smallest_eval}. We now prove \eqref{claim:Ax_inv_spec_est_2} by induction on $|\gamma|$. For any multi-index $0 < |\gamma|\leq m$, differentiation of the identity
$A(x)^{-1}A(x)=I$ yields
\[
0 = D^\gamma(A(x)^{-1}A(x))
= \sum_{\beta < \gamma} \binom{\gamma}{\beta}
D^\beta(A(x)^{-1})\,D^{\gamma-\beta}A(x) + D^\gamma(A(x)^{-1}) A(x),
\]
and hence
\[
D^\gamma(A(x)^{-1})
= -\sum_{\beta < \gamma} \binom{\gamma}{\beta}
D^\beta(A(x)^{-1})\,D^{\gamma-\beta}(A(x))\,A(x)^{-1}.
\]
For the base case $|\gamma|=1$, we have
\[
D^\gamma(A(x)^{-1}) = -A(x)^{-1}D^\gamma A(x)A(x)^{-1},
\]
and thus
\[
\Vert D^\gamma(A(x)^{-1}) \Vert_2
\leq \Vert A(x)^{-1} \Vert_2^2\,\Vert D^\gamma A(x) \Vert_2
\leq \tau_*^{-2} C_{\gamma}\,\delta^{-|\gamma|}\,
\nu^{|\gamma|-m-d}.
\]
We continue with the induction step. Assume that
\eqref{claim:Ax_inv_spec_est_2} holds for all multi-indices with 
$0<|\gamma|\leq s < m$, and let $|\gamma|= s+1$. Then
\begin{align*}
\Vert D^\gamma(A(x)^{-1}) \Vert_2
&\leq \Vert A(x)^{-1} \Vert_2^2\,\Vert D^\gamma A(x) \Vert_2 \\
&\quad + \sum_{0<\beta<\gamma} \binom{\gamma}{\beta} 
\Vert D^\beta(A(x)^{-1}) \Vert_2\,
\Vert D^{\gamma-\beta}A(x) \Vert_2\,
\Vert A(x)^{-1} \Vert_2 \\
&\leq \tau_*^{-2} C_{\gamma}\,\delta^{-|\gamma|}\,
\nu^{|\gamma|-m-d} \\
&\quad + \sum_{0<\beta<\gamma} \binom{\gamma}{\beta}
\wt{C}_{\beta} C_{\gamma-\beta}\tau_*^{-1}\,
\delta^{-|\gamma|}\,\nu^{|\gamma|-2m-d} \\
&= \biggl(
\tau_*^{-2} C_{\gamma}
+ \sum_{0<\beta<\gamma} \binom{\gamma}{\beta}
\wt{C}_{\beta} C_{\gamma-\beta}\tau_*^{-1}\nu^{-m}
\biggr)
\delta^{-|\gamma|}\,\nu^{|\gamma|-m-d}.
\end{align*}
This completes the proof.
\end{proof}
\end{lemma}
We will use the last lemma particularly for bounding
$\lambda(x)=A(x)^{-1}\omega_0$. To be more precise, it immediately
follows that we have for all $x\in\R^d$,
\begin{equation}\label{lambdabound1}
  \|\lambda(x)\| \le \tau_*^{-1}  \nu^{-d} \|\omega_0\|,
  \end{equation}
and, for $0<|\gamma|\le m$,
\begin{equation}\label{lambdabound2}
  \|D^\gamma\lambda(x)\| \le \widetilde{C}_\gamma \delta^{-|\gamma|}
  \nu^{|\gamma|-m-d}\|\omega_0\|.
    \end{equation}
\subsection{Convergence and Super-convergence}\label{sec:superconvergence}

If polynomials of degree $r \in \N_0$ are reproduced, classical
MLS converges with a rate of $r+1$ provided that the target function
$f : \R^d \to \R$ lies in $C^{r+1}(\R^d)$. To prove this well-known
fact in our setting, let us introduce some notation. Let $\cR_r
f:\R^d\times\R^d\to\R$ denote the Taylor remainder of order $r+1$,
defined by 
\begin{equation}\label{taylor_residue_basic}
\cR_r f(y,x) := f(y) - \sum_{|\alpha|\leq r} \frac{D^\alpha f(x)}{\alpha!}\,(y-x)^\alpha
\end{equation}
for all $x,y \in \R^d$. Equivalently, $\cR_r f$ admits the integral representation
\begin{equation}\label{taylor_residue_integral}
\cR_r f(y,x) = \sum_{|\alpha|=r+1}\frac{r+1}{\alpha!}\int_0^1 (1-t)^{r}(y-x)^\alpha D^\alpha f(x+t(y-x))\,dt.
\end{equation}
Both representations will be used in the remainder of this paper. 

\begin{theorem}\label{thm:classic_error}
Suppose $\nu \geq \nu_*$ and $f \in C^\ell(\R^d)$. Then there exists a
constant $C = C(\cT)>0$, independent of $\nu$, such that the error
$E_hf = f-Q_hf$ satisfies the bound
\[
|E_h f(x) |\leq C\,\delta^{\min\{r+1,\ell\}}\, |f |_{W^{\min\{r+1,\ell\},\infty}(B_\delta(x))}
\]
for all $x \in \R^d$.
\begin{proof}
Fix $x \in \R^d$. We only show the case $\ell \geq r+1$ to simplify notation, since the argument for $\ell < r+1$ works analogously. By the polynomial reproduction property of MLS,
\begin{align*}
-E_h f(x) 
&= \sum_{q\in\mathbb{Z}^d} w_\delta(hq,x)\,(f(hq)-f(x)) \\
&= \sum_{q\in\mathbb{Z}^d} w_\delta(hq,x)
\biggl(
\sum_{0<|\alpha|\leq r} \frac{D^\alpha f(x)}{\alpha!}(hq-x)^\alpha + \cR_r f(hq,x)
\biggr) \\
&= \sum_{q\in\mathbb{Z}^d} w_\delta(hq,x)\,\cR_r f(hq,x).
\end{align*}
For every $y \in B_\delta(x)$, the integral representation of the remainder yields
\begin{align*}
|\cR_r f(y,x) |
&\leq 
\sum_{|\alpha|=r+1}\frac{r+1}{\alpha!}\int_0^1 (1-t)^{r}\Vert y-x \Vert^{r+1}|D^\alpha f(x+t(y-x)) |\,dt \\ 
&\leq 
\biggl(\sum_{|\alpha|=r+1}\frac{1}{\alpha!}\biggr)\delta^{r+1}\,|f |_{W^{r+1,\infty}(B_\delta(x))}.
\end{align*}
Moreover, using the bound (\ref{lambdabound1}) on $\lambda(x)$ leads to 
\begin{align*}
|w_\delta(y,x) |
&\leq 
\tau_*^{-1} \nu^{-d} \Vert \omega_0 \Vert \biggl|\sum_{j=1}^M \Phi\biggl(\frac{x-y}{\delta}\biggr)p_j\biggl(\frac{x-y}{\delta}\biggr)\biggr|\\ 
&\leq 
M\, \tau_*^{-1} \nu^{-d} \Vert \omega_0 \Vert \max_{1\leq j\leq M} \Vert \Phi p_j \Vert_{L^\infty(B_1(0))}.
\end{align*}
Finally, observe that
\begin{align}\label{bound_for_index_set}
|X_h \cap B_\delta(x) |
\leq \frac{\textup{vol}_d(B_{\delta+\frac{h}{2}}(0))}
{\textup{vol}_d(B_{\frac{h}{2}}(0))}
= (1+2\nu)^d < 3^d \nu^d.
\end{align}
Collecting the above estimates yields the desired result.
\end{proof}
\end{theorem}

The classical theory we just presented leaves no reason to assume that
any rates higher than $r+1$ can be achieved. However, if $r$ is even,
$f \in C^{r+2}(\R^d)$ and the data sites form a regular grid, a convergence rate of $r+2$ is usually
observed numerically. This is the case because the MLS weights
"almost" reproduce polynomials of degree $r+1$, as the following
theorem shows. 

\begin{theorem}\label{thm:superconv}
Let $w_\delta : \R^d \times \R^d \to \R$ denote the MLS weight
functions for reproduction degree $r \in \mathbb{N}_0$ and suppose the
kernel satisfies $\Phi = \Phi(-\cdot)$ in addition to the properties
from Assumption \ref{general_assumption}. Then, there exist constants $\wt\nu = \wt\nu(d,r)
> 0$ and $C_\alpha = C_\alpha(\Phi,d,m,r)>0$ such that 
\begin{equation}\label{claim:superconv}
\biggl|\sum_{q\in\Z^d} w_\delta(hq,x)(x-hq)^\alpha\biggr|
\leq 
C_\alpha \delta^{|\alpha|} \nu^{1-m},
\end{equation}
provided that $\nu \geq \wt\nu$ and $|\alpha|$ is odd.
\end{theorem}
\begin{proof}
Of course, only the case $|\alpha|\geq r+1$ is relevant, as the
estimate is a trivial consequence of the polynomial reproduction if
$|\alpha|\leq r$. Since the weight functions do not depend on the
basis choice, we may choose $\cB$ to be a monomial basis without
changing the left hand side of~\eqref{claim:superconv}. Set $\beta^1 =
0$ and write  
\[
\{\beta \in \mathbb{N}_0^d : |\beta |\leq r\} = \{\beta^1,\beta^2,\ldots,\beta^M\},
\]
then define $\cB := \{p_1,\ldots,p_M\}$ with $p_j(x) = x^{\beta^j}$ and
set $\wt\nu := \nu_*$ to the constant from~\eqref{Ax_smallest_eval} corresponding to this basis. For this particular basis, it holds
that $\omega_0 = e_1$. Using the symmetry of both the lattice and the kernel
function, we see that the matrix entries $a_{ij}(x)$ from
(\ref{matrixA}) satisfy
\begin{align*}
a_{ij}(-x) &= \sum_{q\in\Z^d}
\Phi\biggl(\frac{-x-hq}{\delta}\biggr)\biggl(\frac{-x-hq}{\delta}\biggr)^{\beta^i
  + \beta^j} \\  
&= 
\sum_{q\in\Z^d}
\Phi\biggl(\frac{hq-x}{\delta}\biggr)\biggl(\frac{hq-x}{\delta}\biggr)^{\beta^i
  + \beta^j} = (-1)^{|\beta^i|+ |\beta^j|} a_{ij}(x). 
\end{align*}
Therefore, writing $D := \textup{diag}\bigl((-1)^{|\beta^1|},
(-1)^{|\beta^2|}, \ldots, (-1)^{|\beta^M|}\bigr) \in \R^{M \times M}$
this gives
\[
A(-x) = D A(x)D
\]
and consequently
\[
\lambda_j(-x) = e_j^T A(-x)^{-1} e_1 = (-1)^{|\beta^1|+ |\beta^j|}
e_j^T A(x)^{-1} e_1 = (-1)^{|\beta^j|} \lambda_j(x). 
\] 
In particular, due to the fact that $\lambda_j$ is $h$-periodic in
every component, we have $\lambda_j(hq) = \lambda_j(0) = 0$ provided
that $|\beta^j|$ is odd. Using this fact, this leads to
\begin{align}\label{super_reform}
\begin{split}
\sum_{q\in\Z^d} w_\delta(hq,x)(x-hq)^\alpha 
&= 
\delta^{|\alpha|}\sum_{j=1}^M \sum_{q\in\Z^d} (\lambda_j(x)-\lambda_j(hq)) \,\Phi\biggl(\frac{x-hq}{\delta}\biggr)\biggl(\frac{x-hq}{\delta}\biggr)^{\beta^j+\alpha} \\ 
&+ 
\delta^{|\alpha|}\sum_{j \in \{1,\ldots,M\} \,:\, |\beta^j|\text{ even}} \lambda_j(0) \sum_{q\in\Z^d} \Phi\biggl(\frac{x-hq}{\delta}\biggr)\biggl(\frac{x-hq}{\delta}\biggr)^{\beta^j+\alpha}.
\end{split}
\end{align}
We begin by estimating the second sum in \eqref{super_reform} and claim that there exists a constant $C_\alpha'>0$ such that
\begin{equation}\label{super_est_sum_1}
\biggl|\sum_{j \in \{1,\ldots,M\} \,:\, |\beta^j|\text{ even}} \lambda_j(0) \sum_{q\in\Z^d} \Phi\biggl(\frac{x-hq}{\delta}\biggr)\biggl(\frac{x-hq}{\delta}\biggr)^{\beta^j+\alpha}\biggr|\leq C_\alpha' \nu^{-m}.
\end{equation}
To see this, note that for every $\beta^j$ with even length, the symmetry of $\Phi$ implies that
\[
\int_{\R^d} \Phi\biggl(\frac{x-y}{\delta}\biggr)\biggl(\frac{x-y}{\delta}\biggr)^{\beta^j+\alpha} \, dy = 0,
\]
as $|\beta^j| + |\alpha|$ is odd. Thus we may apply Proposition~\ref{prop:rav} to obtain 
\[
\biggl|\sum_{q\in\Z^d} \Phi\biggl(\frac{x-hq}{\delta}\biggr)\biggl(\frac{x-hq}{\delta}\biggr)^{\beta^j+\alpha}\biggr|\leq C\nu^{d-m}
\]
for a suitable constant $C>0$ depending on $\Phi$ as well as
$d,m,\beta^j$ and $\alpha$. As $|\lambda_j(0)|$ is bounded from above
by a constant proportional to $\nu^{-d}$, this
proves~\eqref{super_est_sum_1}. Next, we tend to the first sum in
\eqref{super_reform} and show that there exists a constant
$C_\alpha''>0$ such that 
\begin{equation}\label{super_est_sum_2}
\biggl|\sum_{j=1}^M \sum_{q\in\Z^d} (\lambda_j(x)-\lambda_j(hq))
\,\Phi\biggl(\frac{x-hq}{\delta}\biggr)\biggl(\frac{x-hq}{\delta}\biggr)^{\beta^j+\alpha}
\biggr|\leq C_\alpha'' \nu^{1-m}. 
\end{equation}
Recall (\ref{lambdabound2}), 
from which we know that there exists a constant $\wt{C}>0$ with 
\[
|\lambda_j(x) - \lambda_j(hq) |\leq \sup_{y\in\R^d}\|\nabla
\lambda_j(y)\|\Vert x - hq\Vert\leq \wt{C} \nu^{1-m-d}
\,\biggl\Vert\frac{x-hq}{\delta}\biggr\Vert. 
\]
Since the supremum norm of the map $y \mapsto \Phi(y) \Vert
y\Vert^{r+1+|\alpha|}$ and $\nu^{-d}|X_h \cap B_\delta(x) |$ are
bounded by constants, we obtain~\eqref{super_est_sum_2}. Starting
from~\eqref{super_reform} and combining both~\eqref{super_est_sum_2}
and~\eqref{super_est_sum_1} yields the claim. 
\end{proof}

This allows us to state and prove the following improved convergence
result for the standard MLS scheme. 

\begin{theorem}\label{thm:classic_error_improved}
Suppose $r \in \mathbb{N}_0$ is even and $f \in C^{r+2}(\R^d)$. If $\Phi = \Phi(-\cdot)$ and $\nu \geq \wt\nu$, where $\wt\nu$ is the constant from the previous theorem, then there exists $C = C(\Phi,d,m,r)>0$ such that the error $E_hf=f-Q_hf$ of the standard MLS method satisfies
\[
|E_h f(x) |\leq C\bigl( \nu^{1-m}\delta^{r+1}\, |f
|_{W^{r+1,\infty}(B_\delta(x))} + \delta^{r+2} |f
|_{W^{r+2,\infty}(B_\delta(x))}\bigr) 
\]
for all $x \in \R^d$. In particular, if $\nu^{-m} \leq h$ this means
\[
|E_h f(x) |\leq 2 C \delta^{r+2} |f |_{W^{r+1,r+2,\infty}(B_\delta(x))}.
\]
\begin{proof}
Fix $x \in \R^d$. We proceed similarly to the proof of Theorem~\ref{thm:classic_error}, noting that
\[
-E_h f(x) = \sum_{q\in\Z^d} w_\delta(hq,x)\biggl(\sum_{0<|\alpha|\leq r} \frac{D^\alpha f(x)}{\alpha!}(hq-x)^\alpha + \sum_{|\alpha|=r+1} \frac{D^\alpha f(x)}{\alpha!}(hq-x)^\alpha + \cR_{r+1} f(hq,x)\biggr).
\]
Due to the polynomial reproduction property,
\[
\sum_{q\in\Z^d} w_\delta(hq,x) \sum_{0<|\alpha|\leq r} \frac{D^\alpha f(x)}{\alpha!}(hq-x)^\alpha = 0.
\]
Since $r+1$ is odd, Theorem~\ref{thm:superconv} yields 
\begin{align*}
\biggl|\sum_{q\in\Z^d} w_\delta(hq,x)\sum_{|\alpha|=r+1} \frac{D^\alpha f(x)}{\alpha!}(hq-x)^\alpha \biggr|
&\leq 
\sum_{|\alpha|=r+1} \frac{|D^\alpha f(x)|}{\alpha!}\biggl|\sum_{q\in\Z^d} w_\delta(hq,x)(x-hq)^\alpha\biggr|\\ 
&\leq \delta^{r+1}\biggl(\sum_{|\alpha|=r+1}\frac{C_\alpha}{\alpha!}\biggr) \nu^{1-m} |f|_{W^{r+1,\infty}(B_\delta(x))}.
\end{align*}
Finally, just like in the proof of Theorem~\ref{thm:classic_error}, we can estimate
\begin{align*}
\biggl|\sum_{q\in\Z^d} w_\delta(hq,x) \cR_{r+1} f(hq,x) \biggr|
&\leq 
|X_h \cap B_\delta(x) |\sup_{y\in B_\delta(x)} |w_\delta(y,x) ||\cR_{r+1} f(y,x) |\\ 
&\leq 
C \delta^{r+2} |f|_{W^{r+2,\infty}(\R^d)}
\end{align*}
for a suitable constant $C>0$ not depending on $\nu$.
\end{proof}
\end{theorem}

\section{The Multilevel Approach}\label{sec:MLS}

We begin by introducing the multilevel procedure that will be analyzed in this chapter. 
Fix three parameters $h_0 > 0$ and $0 < \mu < 1 < \nu$. As usual,
$h_0$ denotes the initial mesh width,  
 $\mu$ is the refinement factor between two successive levels and
$\nu$ is the ratio between support radius and mesh width. Hence, for each 
$j\in \N_0$, we define
\[
h_j := \mu^j h_0
\qquad\text{and}\qquad
\delta_j := \nu h_j .
\]

The MLS multilevel scheme which we will describe now is a simple error correction
scheme. We construct a sequence of multilevel approximants 
$\cQ_0, \cQ_1, \cQ_2, \ldots$ and the associated sequence of error operators
$\cE_0, \cE_1, \cE_2, \ldots$ by starting with $\cQ_0f:=0$ and
$\cE_0f=f$ and then defining for $j\in\N$ recursively 
\begin{align*}
\cQ_j f 
&:= \cQ_{j-1} f + Q_{h_j}\,\cE_{j-1} f,
\\
\cE_j f 
&:= \cE_{j-1}f - Q_{h_j}\cE_{j-1}f .
\end{align*}

We obviously have  $\cE_j f=  E_{h_j}\,\cE_{j-1} f=E_{h_j} E_{h_{j-1}} \cdots
E_{h_1} f$. Moreover, the following elementary identities hold.

\begin{lemma}
For all $j \in \N_0$, the iterations of the MLS multilevel scheme satisfy
\[
\cE_j f = f - \cQ_j f
\qquad\text{and}\qquad
\cQ_j f = \sum_{i=1}^j Q_{h_i}\,\cE_{i-1} f .
\]
\end{lemma}

\begin{proof}
We prove the first identity by induction. The case $j=0$ follows
directly from the definitions. 
Let $j \in \N$, and assume the statement holds for $j-1$. Then, 
\[
\cE_j f  = \cE_{j-1} f - Q_{h_j}\,\cE_{j-1} f 
= \bigl(f - \cQ_{j-1} f\bigr) - \bigl(\cQ_j f - \cQ_{j-1} f\bigr) = f - \cQ_j f,
\]
where the induction hypothesis and the definition of $\cQ_j$ was used in the second step.
The second stated identity follows immediately from the recursive definition of $\cQ_j$.
\end{proof}

Clearly, evaluating $\cQ_j f$ is computationally more expensive than computing 
$Q_{h_j} f$ alone. However, provided that $\nu$ is sufficiently large and $\mu$ is 
sufficiently small, we will show that for any nonempty bounded domain $\Omega \subset \R^d$ 
the sequence of errors $\|\cE_j f\|_{L^\infty(\Omega)}$ can be expected to
converge to zero substantially faster than  
$\|E_{h_j} f\|_{L^\infty(\Omega)}$ as $j \to \infty$.

\subsection{Error Analysis}\label{sec:error_analysis}

We begin our error analysis of this multilevel MLS method by
introducing a family of semi-norms that will be used throughout this
subsection. 
Let $\varepsilon > 0$, and let $j, n \in \N_0$ be integers with $j \leq n$.
For any nonempty bounded domain $\Omega \subset \R^d$ and any function
$g \in W^{n,\infty}(\Omega)$, we define
\[
|g |_{W^{j,n,\infty}_\varepsilon(\Omega)}
:=
\max_{j \leq |\alpha|\leq n}
\bigl( \varepsilon^{|\alpha|}
      \|D^\alpha g \|_{L^\infty(\Omega)} \bigr),
\]
and, in particular,
\[
|g |_{W^{n,\infty}_\varepsilon(\Omega)}
:=
|g |_{W^{n,n,\infty}_\varepsilon(\Omega)}
= \varepsilon^n |g |_{W^{n,\infty}(\Omega)} .
\]

Let us now outline the strategy of the error analysis. Our goal is to
establish the following recursive result on the MLS error operator. As
usual $B_R(x)$ denotes the open ball about $x$ with radius $R>0$. 

\begin{theorem}\label{thm:main}
There exists a constant $\cC = \cC(\cT, k)>0$ such that for every $f
\in C^k(\mathbb{R}^d)$ and all $h > 0$, $R > 0$, $\nu \geq \nu_*$ and
$r+1 \leq n \leq k$, the estimate 
\begin{equation}\label{claim:main}
|E_h f |_{W^{n,\infty}_{\delta}(B_R(x))}
\leq
\cC\,\bigl(\nu^{n-m} |f|_{W^{r+1,k,\infty}_\delta(B_{R+\delta}(x))} +
|f|_{W^{\min\{n+r+1,k\},\infty}_\delta(B_{R+\delta}(x))} \bigr) 
\end{equation}
holds, where $\delta \defg \nu h$. In particular, for every $\mu\in(0,1]$ we obtain
\begin{equation}\label{claim:main_secondary}
|E_h f|_{W_\delta^{r+1,k,\infty}(B_R(x))} \leq
\cC\,\bigl(\nu^{k-m}\mu^{r+1} + \mu^{\min\{2r+2,k\}}\bigr)
\,|f|_{W_{\delta\slash\mu}^{r+1,k,\infty}(B_{R+\delta}(x))}. 
\end{equation}
\end{theorem}

For the remainder of this subsection, we fix $h > 0$ and $\nu \geq
\nu_*$. Furthermore, the integer $L \in \mathbb{N}$ will count the
number of levels. Assuming that Theorem~\ref{thm:main} holds, one can
show that the multilevel method exhibits superior convergence compared
to the classical MLS approach. To formulate our results more
concisely, let us introduce the two families  
\[
(R_j)_{0\leq j\leq L} \subset \R_{>0} \qquad \text{and} \qquad (\theta_j)_{0\leq j\leq L} \subset \{0,\ldots,k-r-1\}
\]
which are defined via
\[
R_j := R_{j-1} + \delta_{L-j+1} \qquad \text{and} \qquad \theta_{j} := \min\{j(r+1), k - r - 1\}
\]
for $1 \leq j \leq L$, where $R_0 > 0$ is arbitrary and
$\theta_0=0$. These families bear the following meanings. The number
$R_L$ is defined in such a way that evaluating the multilevel
approximant $\cQ_L f$ at a point $x \in \R^d$ only requires the grid
points lying in $X_h \cap B_{R_L-R_0}(x)$. Note that for every  $L \in
\mathbb{N}$, the bound 
\begin{equation}\label{S_L_bound}
S_L := R_L - R_0 = \delta_1 + \cdots + \delta_L \leq \sum_{j=1}^\infty \delta_j = \frac{\delta_1}{1-\mu}
\end{equation}
holds, which will imply that the multilevel method is localized
uniformly in $L$. Therefore, one of the main advantages of the MLS
procedure is maintained. The quantities $\theta_1,\ldots,\theta_L$ only
occur in powers of $\mu$ and essentially measure how many convergence
orders can be gained by the multilevel method as compared to the
classical approach. This will be made more precise in the following
result.  

\begin{lemma}\label{lem:multilevel}
Suppose $f\in C^k(\R^d)$ and $\cC=\cC(\cT, k)$ is the constant from
Theorem~\ref{thm:main}. Then we obtain 
\begin{equation}\label{true_recursion}
|\cE_L f |_{W^{r+1,\infty}_{\delta_L}(B_{R_0}(x))} \leq \cC^L \,
\mu^{L(r+1)} \biggl(\mu^{\theta_1+\cdots+\theta_L} +
\nu^{k-m}\sum_{j=0}^{L-1}\mu^{\theta_0+\cdots+\theta_j} \rho^{L-j-1}
\biggr) |f|_{W^{r+1,k,\infty}_{\delta_0}(B_{R_L}(x))} 
\end{equation}
for all $L \in \mathbb{N}$ and $x \in \R^d$, where $\rho := \nu^{k-m} +
\mu^{\theta_1}$. The simpler bound 
\begin{equation}\label{true_recursion_simplified}
|\cE_L f |_{W^{r+1,\infty}_{\delta_L}(B_{R_0}(x))} \leq  \cC^L
\mu^{L(r+1)} \rho^L |f|_{W^{r+1,k,\infty}_{\delta_0}(B_{R_L}(x))} 
\end{equation}
holds as well.
\end{lemma}

Lemma~\ref{lem:multilevel} is a direct consequence of
Theorem~\ref{thm:main}, and we will prove both results towards the end
of Subsection~\ref{sec:proofs}. Estimates~\eqref{true_recursion}
and~\eqref{true_recursion_simplified} motivate the definition of  
\begin{equation}\label{def:J_L}
J_L(\nu,\mu) := \cC \min\biggl\{\rho,
\biggl(\mu^{\theta_1+\cdots+\theta_L} +
\nu^{k-m}\sum_{j=0}^{L-1}\mu^{\theta_0+\cdots+\theta_j} \rho^{L-j-1}
\biggr)^{1\slash L}\biggr\}, 
\end{equation}
which precisely quantifies the speedup obtained by the multilevel
method. In the following theorem, which is the main result of this section, we use 
the the notation $B_R(\Omega)$ to denote the set of all $x \in \R^d$
for which there is a $y \in \Omega$ such that $\Vert x-y\Vert<
\delta$.

\begin{theorem}\label{thm:result}
Let $\Omega \subset \R^d$ be a nonempty and bounded open set. Then
there exists a constant $C_0 = C_0(\cT)>0$ such that, for all integers
$L \geq 2$ and all $f \in C^k(\mathbb{R}^d)$, 
\begin{equation}\label{claim:result_1}
\Vert\cE_L f\Vert_{L^\infty(\Omega)} \leq \frac{C_0}{h_0^{r+1}} \bigl(J_{L-1}(\nu,\mu)\bigr)^{L-1} h_L^{r+1} |f|_{W^{r+1,k,\infty}_{\delta_0}(B_{S_L}(\Omega))}.
\end{equation}
If $\Omega$ is a Lipschitz domain and $h_L = \mu^L h_0 < 1$, we additionally obtain
\begin{equation}\label{claim:result_2}
\|\cE_L f\|_{W^{n,\infty}(\Omega)} \leq \frac{C_0}{h_0^{r+1}} \bigl(J_{L-1}(\nu,\mu)\bigr)^{L-1} \,|f|_{W^{r+1,k,\infty}_{\delta_0}(B_{S_L}(\Omega))}.
\end{equation}
\begin{proof}
By Theorem~\ref{thm:classic_error}, there exists a constant $C_0$ as stated such that
\begin{align*}
\begin{split}
\Vert\cE_{L} f \Vert_{L^\infty(\Omega)}
&= \Vert E_{h_{L}} \cE_{L-1} f \Vert_{L^\infty(\Omega)}
\leq
C_0 |\cE_{L-1} f|_{W^{r+1,\infty}_{\delta_{L}}(B_{\delta_{L}}(\Omega))}  \\ &\leq C_0\, \mu^{r+1} |\cE_{L-1} f|_{W^{r+1,\infty}_{\delta_{L-1}}(B_{\delta_{L}}(\Omega))}.
\end{split}
\end{align*}
Hence~\eqref{claim:result_1} follows immediately from Lemma~\ref{lem:multilevel} with $R_0 := \delta_L$. To prove~\eqref{claim:result_2}, we combine Theorem~\ref{thm:classic_error} with the recursion estimate~\eqref{claim:main} in the case $k=n=r+1$ to see that
\[
\Vert\cE_L f \Vert_{L^\infty(\Omega)} + |\cE_L f |_{W_{\delta_L}^{r+1,\infty}(\Omega)} \leq C |\cE_{L-1} f|_{W_{\delta_L}^{r+1,\infty}(B_{\delta_L}(\Omega))}
\]
for a suitable constant $C>0$. The additional assumption that $\Omega \subset \R^d$ is a Lipschitz domain implies that the norm 
\[
\|\cdot\|:= \|\cdot \|_{L^\infty(\Omega)} + |\cdot |_{W^{r+1,\infty}(\Omega)}
\]
is equivalent to the classical Sobolev norm $\|\cdot\|_{W^{r+1,\infty}(\Omega)}$ on $W^{r+1,\infty}(\Omega)$. This follows from standard interpolation inequalities for Sobolev spaces, c.f. for example Corollary 13.62 in~\cite{Leoni-17-1}. 
Hence there exists an equivalence constant $C' = C'(\Omega,d,r)>0$ such that
\begin{align*}
h_{L}^{r+1} \|\cE_L f\|_{W^{r+1,\infty}(\Omega)} &\leq C'(\|\cE_L f
\|_{L^\infty(\Omega)} + |\cE_L f
|_{W_{\delta_L}^{r+1,\infty}(\Omega)}) \\ &\leq C' C |\cE_{L-1}
f|_{W_{\delta_L}^{r+1,\infty}(B_{\delta_L}(\Omega))} \\ &\leq C' C
\mu^{r+1} |\cE_{L-1}
f|_{W_{\delta_{L-1}}^{r+1,\infty}(B_{\delta_{L}}(\Omega))}. 
\end{align*}
Applying Lemma~\ref{lem:multilevel} to the right hand side of this estimate and dividing by $h_L^{r+1} = \mu^{L(r+1)} h_0^{r+1}$ yields the claim. 
\end{proof}
\end{theorem}

Note that semi-norms of $f$ occurring on the right hand sides depend
on $\nu$ over $\delta_0$ and $S_L$. However, this dependence does not influence
the convergence rate, since it occurs outside of the convergence
enforcing terms. In the case of $h_0$ being small enough such that
$\delta_0 = \nu h_0 \leq 1$, we may even estimate the weighted norm
against the classical one. In the situation of
Theorem~\ref{thm:result}, the classical MLS estimate from
Theorem~\ref{thm:classic_error} reads 
\[
\|E_{h_L} f \|_{L^\infty(\Omega)} \leq \frac{C_0}{h_0^{r+1}}\,
h_L^{r+1} |f|_{W_{\delta_0}^{r+1,\infty}(B_{\delta_L}(\Omega))}, 
\]
guaranteeing a convergence rate of at least $r+1$. Comparing this
to~\eqref{claim:result_1}, we immediately see that the factor
containing a power of $J_{L-1}(\nu,\mu)$ is the only significant
difference in the case of $L \geq 2$. Since the quantity $J_{L-1}(\mu,\nu)$ becomes arbitrarily small as $\nu \to \infty$ and $\mu \to 0$, we
have shown that the multilevel method will always accelerate the convergence if $\nu$ and $\mu$ are chosen suitably. Of course, the choice of these parameters may be subject to practical constraints. There is no difference in the case of $L=1$ since the standard MLS method and the new multilevel method coincide if the number of levels of the latter is one. 

Moreover, using a fractional Gagliardo-Nirenberg inequality of the form
\[
\|f\|_{W^{s,\infty}(\Omega)} \le C
\|f\|_{W^{r+1,\infty}(\Omega)}^{\frac{s}{r+1}}
\|f\|_{L^\infty(\Omega)}^{1-\frac{s}{r+1}}, \qquad s \in (0,r+1),
\]
which holds, for example,  if $\Omega$ is a bounded, open Lipschitz domain
(see Theorem 1 in \cite{Brezis-Mironescu-18-1}),  allows us to combine
(\ref{claim:result_1}) and (\ref{claim:result_2}) to derive the
simultaneous approximation result
\[
\Vert\cE_L f\Vert_{W^{s,\infty}(\Omega)} \leq \frac{CC_0}{h_0^{r+1}}
\bigl(J_{L-1}(\nu,\mu)\bigr)^{L-1} h_L^{r+1-s}
|f|_{W^{r+1,k,\infty}_{\delta_0}(B_{S_L}(\Omega))}, \qquad s \in \lbrack 0, r+1 \rbrack.
\]
Comparing this  again to  the corresponding result for classic MLS,
which is essentially given by
\[
\|E_{h_L} f\|_{W^{s,\infty}(\Omega)} \le \frac{C_0}{h_0^{r+1}} h_L^{r+1-s}
\|f\|_{W_{\delta_0}^{r+1,\infty}(B_{\delta_L(\Omega)})},\qquad s \in \lbrack 0, r+1\rbrack
\]
after applying a similar interpolation argument (see for example \cite{Mirzaei-15-1}), we see that the multilevel MLS result does not only yield convergence for derivatives up to order $r+1$ instead of $r$, provided that $J_{L-1}(\nu,\mu)<1$, but also provides improved convergence  for derivative information up to order $r$.

Before we turn to the proofs, let us briefly discuss how many orders
of convergence can be gained via the multilevel method according to
our results. To this end, we introduce the ratio 
\[
\kappa := \nu^{k-m}\slash\mu^{\theta_1}
\]
which measures how the quantities $\nu^{k-m}$ and $\mu^{\theta_1}$ are related
to each other, noting also that
$\rho=\mu^{\theta_1}(1+\kappa)$. If $\eta>0$ is chosen such that 
$\cC(1+\kappa)\mu^{\eta}=1$, the convergence improving term can be
bounded by 
\[
\bigl(J_{L-1}(\nu,\mu)\bigr)^{L-1} \leq (\cC\rho)^{L-1} = \mu^{L(1-\frac{1}{L})(\theta_1-\eta)},
\]  
which means that $L$ multilevel steps reduce the error at least with order 
\[
\sigma = \sigma_L := r+1+\bigl(1-\frac{1}{L}\bigr)(\theta_1-\eta).
\]
Operating under the assumption that $\eta $ is small and $L$ is large,
this means that a gain of almost $\theta_1 = \min\{r+1,k-r-1\}$ orders
can be achieved compared to the standard MLS method, essentially
doubling the convergence rate in the case of $k=2r+2$. If $k \geq
L(r+1)$, we can obtain a significantly sharper bound for
$J_{L-1}(\nu,\mu)$. 

\begin{lemma}\label{lem:convergence_gain}
Let $L \geq 2$ and suppose $k \geq L(r+1)$. Then the estimate
\begin{equation}\label{claim:convergence_gain}
J_{L-1}(\nu,\mu) \leq
\min\biggl\{1,\biggl\lbrack\frac{\mu^{\frac{1}{2}(L-2)(L-1)(r+1)}}{(1+\kappa)^{L-1}} +  \biggl(\frac{\kappa}{1+\kappa}\biggr)\biggl(\frac{2-\mu}{1-\mu}\biggr)\biggr\rbrack^{1\slash(L-1)}\biggr\} \cdot \cC\rho.
\end{equation}
holds.
\end{lemma}
A proof of this lemma can be found at the end of Section~\ref{sec:proofs}. To interpret this result, observe that if $\kappa$ is small, the map
\[
L \mapsto \min\left\{1,\left[\frac{\mu^{\frac{1}{2}(L-2)(L-1)(r+1)}}{(1+\kappa)^{L-1}} +  \left(\frac{\kappa}{1+\kappa}\right)\left(\frac{2-\mu}{1-\mu}\right)\right]^{1/(L-1)}\right\}
\]
evaluates to approximately to $1$ both at $L = 2$ and as $L \to
\infty$. For intermediate level numbers, however, it attains
significantly smaller values, which yields an additional boost in
convergence.  

Qualitatively, this suggests the following worst-case behavior: for very small $L$, only a reduction rate of around $\sigma$ can be guaranteed. The speed then gradually increases over the first few levels. Because the right-hand side of estimate~\eqref{claim:convergence_gain} converges to $\cC\rho$ as $L \to \infty$ (and since~\eqref{claim:convergence_gain} is strictly valid only for finite $L$), the convergence rate eventually slows down and plateaus around $\sigma$ once again. This phenomenon becomes increasingly pronounced as the ratio $\kappa$ decreases.

Of course, applying the multilevel scheme may also be detrimental to
the convergence speed if $\theta_1 - \eta < 0$, which, at least
theoretically, can always be avoided if $\mu$ is sufficiently
small and $k > r+1$. Furthermore, the abstract constant $\cC$ from
Theorem~\ref{thm:main} generally grows in the smoothness parameters
$k$ and $m$. To obtain exact rates, one would have to treat the
constants occurring in our error estimates more carefully and obtain
concrete bounds for $\eta$. 

\subsection{Proofs of the Error Bounds}\label{sec:proofs}

The aim of this subsection is to establish proofs for
Theorem~\ref{thm:main} and Lemmas~\ref{lem:multilevel} and
\ref{lem:convergence_gain}. To this end, 
we decompose the error into three components and apply the 
triangle inequality. Let us introduce an auxiliary function that will
occur frequently in the subsequent analysis. For $x,y \in
\mathbb{R}^d$ and multi-indices 
$\gamma \in \mathbb{N}_0^d$ with $|\gamma|\leq m$, we define
\[
G_\gamma(y,x)
\defg
\sum_{j=1}^M
D_x^\gamma
\biggl[
\Phi\!\biggl( \frac{x-y}{\delta} \biggr)
p_j\!\biggl( \frac{x-y}{\delta} \biggr)
\biggr]
\lambda_j(x).
\]
For $|\gamma|\leq k$, we set
\begin{align*}
\varepsilon_\gamma^{\mathrm{prod}}(x)
&\defg
D^\gamma Q_hf(x)
-
\sum_{q \in \mathbb{Z}^d} G_\gamma(hq,x) \cR_r f(hq,x), \\
\varepsilon_\gamma^{\mathrm{quad}}(x) 
&\defg
\sum_{q \in \mathbb{Z}^d} G_\gamma(hq,x) \cR_r f(hq,x)
-
Q_hD^\gamma f (x),
\end{align*}
where the Taylor remainder $\cR_r f$ or order $r+1$ is defined as
in~\eqref{taylor_residue_basic}. With this notation, the error can be
written as 
\[
D^\gamma E_h f(x)
=
E_h D^\gamma f(x)
-
\varepsilon_\gamma^{\mathrm{quad}}(x)
-
\varepsilon_\gamma^{\mathrm{prod}}(x).
\]
Note that $G_0(y,x) = w_\delta(y,x)$ and 
\[
\sum_{q \in \mathbb{Z}^d} w_\delta(hq, x) D^\gamma f(hq) = Q_h D^\gamma f(x).
\]

The guiding principle behind this decomposition is to successively isolate terms that decay in $\nu$, up to weighted semi-norms of $f$, and can therefore be absorbed into the right-hand side of~\eqref{claim:main} without altering the overall estimate.
The terminology reflects the nature of each contribution. The term
$\varepsilon_\gamma^{\mathrm{prod}}$ arises from the multivariate
product rule, while $\varepsilon_\gamma^{\mathrm{quad}}$ essentially
denotes a quadrature error and can be estimated using
Proposition~\ref{prop:rav}. We begin by examining
$\varepsilon_\gamma^{\mathrm{prod}}$. Since the weight functions $w_q$
admit the representation 
\begin{equation}\label{weightfct_recall}
w_q = w_\delta(hq,\cdot) = \sum_{j=1}^M \Phi\!\biggl(
\frac{\cdot-hq}{\delta} \biggr)p_j\!\biggl( \frac{\cdot-hq}{\delta}
\biggr) 
\lambda_j
\end{equation}
with $\lambda_j(x) = e_j^T A(x)^{-1} \omega_0$, differentiation of
$w_q$ produces two types of terms. Those involving derivatives of
$\lambda_j$ decay rapidly in $\nu$ by~\eqref{lambdabound2},
while the remaining contributions are collected in $G_\gamma$. We will
see this concretely in the next result.  

\begin{lemma}\label{lem:E1}
For every multi-index $\gamma \in \mathbb{N}_0^d$ with
$r+1 \leq |\gamma |\leq k$, there exists a constant
$C_{1,\gamma} = C_{1,\gamma}(\cT, k)>0$ such that, for all
$x \in \mathbb{R}^d$,
\begin{equation}
|\varepsilon_\gamma^{\mathrm{prod}}(x) |
\leq
C_{1,\gamma}\,\delta^{-|\gamma|}\,\nu^{|\gamma|-m}\,|f |_{W^{r+1,\infty}_\delta(B_\delta(x))}.
\end{equation}
\begin{proof}
By Lemma~\ref{lem:weightfct_mixed_derivatives}, we may replace point evaluations of $f$ by the local
remainder, obtaining
\[
D^\gamma Q_hf(x) = 
\sum_{q \in \mathbb{Z}^d} D^\gamma w_q(x)\, f(hq)
=
\sum_{q \in \mathbb{Z}^d} D^\gamma w_q(x)\, \cR_r f(hq,x).
\]
Consequently,
\begin{align*}
\varepsilon_\gamma^{\mathrm{prod}}(x) 
&=
\sum_{q \in \mathbb{Z}^d \,:\, hq \in B_\delta(x)}
\biggl(
D^\gamma w_q(x)
-
G_\gamma(hq,x)
\biggr)
\cR_r f(hq,x).
\end{align*}
We first show that there exists a constant
$c_\gamma = c_\gamma(\cT, k)>0$ such that, for every
$q \in \mathbb{Z}^d$,
\begin{equation}\label{E1_preliminary_est_1}
\bigl|D^\gamma w_q(x) - G_\gamma(hq,x) \bigr|\leq c_\gamma\,\delta^{-|\gamma|}\,\nu^{|\gamma|-m-d}.
\end{equation}
To this end, we use~\eqref{weightfct_recall} and apply the product rule, which yields 
\begin{align*}
D^\gamma w_q(x) 
&= 
\sum_{j=1}^M \sum_{\beta \leq \gamma}\binom{\gamma}{\beta}D^{\gamma-\beta}\!\biggl[
\Phi\!\biggl(\frac{x-hq}{\delta}\biggr) p_j\!\biggl(\frac{x-hq}{\delta}\biggr)
\biggr]
D^\beta \lambda_j(x) \\
&= 
G_\gamma(x,hq) +  \sum_{0<\beta \leq \gamma}\binom{\gamma}{\beta}\sum_{j=1}^M D^{\gamma-\beta}\!\biggl[
\Phi\!\biggl(\frac{x-hq}{\delta}\biggr) p_j\!\biggl(\frac{x-hq}{\delta}\biggr)\biggr]D^\beta \lambda_j(x).
\end{align*}
The term $G_\gamma(hq,x)$ cancels in~\eqref{E1_preliminary_est_1}. Using~\eqref{lambdabound2} to bound the derivatives of $\lambda$ and estimate the remaining sum, we obtain
\[
\bigl|D^\gamma w_q(x) - G_\gamma(hq,x) \bigr|\leq  \sum_{0<\beta \leq \gamma}\binom{\gamma}{\beta} \sum_{j=1}^M \delta^{-|\gamma-\beta|}|\Phi p_j |_{W^{k,\infty}(\R^d)} \cdot\|\omega_0\|\wt{C}_\beta \delta^{-|\beta|} \nu^{|\beta|-m-d},
\]
which implies~\eqref{E1_preliminary_est_1}. Finally, using that $\nu^{-d} |X_h \cap B_\delta(x)|$ is bounded by a constant (c.f.~\eqref{bound_for_index_set}) and the integral form of the Taylor remainder, there exists a constant $c' = c'(d,r)>0$ such that
\begin{equation}\label{E1_preliminary_est_2}
\sum_{q \in \mathbb{Z}^d \,:\, hq \in B_\delta(x)}
|\cR_r f(hq,x) |
\leq
c'\, \nu^d
|f |_{W^{r+1,\infty}_\delta(B_\delta(x))}.
\end{equation}
Combining~\eqref{E1_preliminary_est_1} and~\eqref{E1_preliminary_est_2} completes the proof.
\end{proof}
\end{lemma}

We continue with the term $\varepsilon_\gamma^{\mathrm{quad}}$, which represents a quadrature error. Before we proceed, let us motivate why this is the case. Note that for any $r+1\leq|\gamma|\leq k$, repeated integration by parts yields 
\begin{align*}
\int_{\mathbb{R}^d} G_\gamma(y,x)\cR_r f(y,x)\,dy
&= \int_{\mathbb{R}^d} \sum_{j=1}^M
   D_x^\gamma \!\left[
      \Phi\!\left(\frac{x-y}{\delta}\right)
      p_j\!\left(\frac{x-y}{\delta}\right)
   \right] \lambda_j(x)\, \cR_r f(y,x)\,dy \\
&= (-1)^{|\gamma|}
   \int_{\mathbb{R}^d} \sum_{j=1}^M
   D_y^\gamma \!\left[
      \Phi\!\left(\frac{x-y}{\delta}\right)
      p_j\!\left(\frac{x-y}{\delta}\right)
   \right] \lambda_j(x)\, \cR_r f(y,x)\,dy \\
&= \int_{\mathbb{R}^d} w_\delta(y,x) D_y^\gamma \cR_r f(y,x)\,dy 
= \int_{\mathbb{R}^d} w_\delta(y,x) D^\gamma f(y)\,dy
\end{align*}
and therefore
\begin{align}\label{quadrature_split}
\begin{split}
\bigl|\epsilon_\gamma^{\mathrm{quad}}(x)\bigr|
&\leq 
\biggl|\sum_{q \in \mathbb{Z}^d} G_\gamma(hq,x) \cR_r f(hq,x) - \frac{1}{h^d}\int_{\mathbb{R}^d} G_\gamma(y,x)\cR_r f(y,x)\,dy \biggr|\\ 
&+ 
\biggl|\sum_{q \in \mathbb{Z}^d} G_0(hq,x) D^\gamma f(hq) - \frac{1}{h^d}\int_{\R^d} G_0(y,x) D^\gamma f(y)\,dy\biggr|.
\end{split}
\end{align}
Hence we may bound $\epsilon_\gamma^{\mathrm{quad}}$ by estimating two quadrature errors. This is again achieved using Proposition~\ref{prop:rav}. We will see that $\varepsilon_\gamma^{\mathrm{quad}}$ also exhibits a favorable behavior as $\nu \to \infty$. 

\begin{lemma}\label{lem:E2}
Let $\gamma \in \mathbb{N}_0^d$ be a multi-index with $r+1 \leq |\gamma|\leq k$. Then there exists a constant $C_{2,\gamma} = C_{2,\gamma}(\cT, k)>0$
such that, for every $x \in \mathbb{R}^d$, 
\begin{equation}
\bigl|\epsilon_\gamma^{\mathrm{quad}}(x)\bigr|
\leq C_{2,\gamma}\, \delta^{-|\gamma|}\, \bigl( \nu^{|\gamma|-m}\,
      |f|_{W^{r+1,k,\infty}_\delta(B_\delta(x))} + |f|_{W^{\min\{|\gamma|+r+1,k\},\infty}_\delta(B_\delta(x))} \bigr).
\end{equation}
\begin{proof}
During this proof, $C > 0$ will denote a generic constant which may change from line to line and depends on $d$, $k$ and $r$. Furthermore, suppose $\gamma$ is fixed and set 
\[
k' := \min\{|\gamma|+r+1,k\}.
\]
We want to estimate both expressions in~\eqref{quadrature_split} using Proposition~\ref{prop:rav}. Let us focus in on the second term first: Taylor expanding $D^\gamma f$ around the point $x \in \R^d$, we obtain
\begin{align*}
\biggl|\sum_{q \in \mathbb{Z}^d} G_0(hq,x) &D^\gamma f(hq) - \frac{1}{h^d}\int_{\R^d} G_0(y,x) D^\gamma f(y)\,dy \biggr| \leq \\ 
&\leq 
\sum_{|\beta|< k'-|\gamma|} \frac{|D^{\gamma+\beta} f(x)|}{\beta!}\biggl|\sum_{q\in\Z^d} G_0(hq,x)(hq-x)^\beta - \frac{1}{h^d}\int_{\R^d} G_0(y,x)(y-x)^\beta\,dy\biggr|\\ 
&+ 
C\delta^{-|\gamma|} |f|_{W_\delta^{k',\infty}(B_\delta(x))} \biggl(\sum_{q\in\mathbb{Z}^d} |G_0(hq,x)|+ \frac{1}{h^d}\int_{\R^d} |G_0(y,x)|\,dy\biggr),
\end{align*}
where the first sum is empty (and hence equal to $0$) if $|\gamma|=k'=k$. Turning to the first term in~\eqref{quadrature_split}, we may rewrite the local remainder $\cR_r f(\cdot,x)$ as
\begin{equation}\label{double_taylor}
\cR_r f(y,x) = \sum_{r<|\beta|< k'} \frac{D^\beta f(x)}{\beta!}(y-x)^\beta + \cR_{k'-1} f(y,x)
\end{equation}
to conclude the estimate
\begin{align*}
\biggl|\sum_{q \in \mathbb{Z}^d} G_\gamma(hq,x) &\cR_r f(hq,x) - \frac{1}{h^d}\int_{\mathbb{R}^d} G_\gamma(y,x)\cR_r f(y,x)\,dy \biggr|
 \\ 
&\leq 
\sum_{r<|\beta|< k'} \frac{|D^\beta f(x)|}{\beta !} \biggl|\sum_{q\in\Z^d} G_\gamma(hq,x)(hq-x)^\beta - \frac{1}{h^d}\int_{\R^d} G_\gamma(y,x)(y-x)^\beta\,dy\biggr|\\ 
&+ 
C |f|_{W^{k',\infty}_\delta(B_\delta(x))}\biggl(\sum_{q\in\Z^d} |G_\gamma(hq,x)|+ \frac{1}{h^d}\int_{\R^d} |G_\gamma(y,x)|\ \,dy\biggr).
\end{align*}
Since $\nu^{-d}|X_h \cap B_\delta(x)|$ is bounded by a constant and
$
|\lambda_j(x)|\leq \tau_*^{-1} \|\omega_0\|\nu^{-d}
$
due to~\eqref{lambdabound1}, it is easy to see that for all
$|\gamma|\leq k$ there exists a constant $c_\gamma = c_\gamma(\cT)$
such that 
\[
\sum_{q\in\mathbb{Z}^d} |G_\gamma(hq,x)|+ \frac{1}{h^d}\int_{\R^d} |G_\gamma(y,x)|\,dy \leq c_\gamma \delta^{-|\gamma|}.
\]
Setting $F_{\beta,\gamma}(\cdot,x) := G_\gamma(\cdot,x)(\cdot-x)^\beta$ for multi-indices $\gamma$, $\beta$ with $|\gamma|\leq k$, $|\beta|\leq k'$ it therefore suffices to estimate the quadrature error
\[
\biggl|\sum_{q\in\Z^d} F_{\beta,\gamma}(hq,x) - \frac{1}{h^d}\int_{\R^d} F_{\beta,\gamma}(y,x)\,dy\biggr|
\]
to bound both terms in~\eqref{quadrature_split}. Since $\Phi$ is compactly supported,
\[
F_{\beta,\gamma}(\cdot,x) \in C_c^{m-|\gamma|}(\R^d).
\]
We may therefore apply Proposition~\ref{prop:rav} to obtain
\begin{align*}
|\epsilon_\gamma^{\mathrm{quad}}(x)|
&\leq 
C^* h^{m} \nu^d \sum_{|\beta|< k'-|\gamma|} \frac{|D^{\gamma+\beta} f(x)|}{\beta!} |F_{\beta,0}(\cdot,x)|_{W^{m,\infty}(\R^d)} + c_0 \,C\delta^{-|\gamma|} |f|_{W_\delta^{k',\infty}(B_\delta(x))} \\
&+ 
C^* h^{m-|\gamma|} \nu^d \sum_{r<|\beta|\leq k'} \frac{|D^\beta f(x)|}{\beta !} |F_{\beta,\gamma}(\cdot,x)|_{W^{m-|\gamma|,\infty}(\R^d)} + c_\gamma\,C \delta^{-|\gamma|} |f|_{W_\delta^{k',\infty}(B_\delta(x))}.
\end{align*}
It remains to show that there exists a family of constants
$C_{\beta,\gamma} = C_{\beta,\gamma}(\cT)>0$ such that
\begin{equation}\label{claim:F_bg}
\bigl|F_{\beta,\gamma}(\cdot,x) \bigr|_{W^{m-|\gamma|,\infty}(\mathbb{R}^d)}
\leq C_{\beta,\gamma}\, \delta^{|\beta|-m}\, \nu^{-d}.
\end{equation}
To this end, let $\xi \in \mathbb{N}_0^d$ with $|\xi|= m-|\gamma|$. We differentiate with respect to $y$, yielding
\begin{align}\label{E2_F_est} 
\begin{split}
|D_y^\xi F_{\beta,\gamma}(y,x)|
&= \biggl|
   \sum_{\chi \leq \xi} \binom{\xi}{\chi}
   \sum_{j=1}^M D_y^{\gamma+\xi-\chi}
   \!\left[
      \Phi\!\left(\frac{x-y}{\delta}\right)
      p_j\!\left(\frac{x-y}{\delta}\right)
   \right]
   \lambda_j(x)\,
   D_y^\chi \lbrack(y-x)^\beta\rbrack
   \biggr|\\
&\leq \wt{C}\, \nu^{-d}
   \sum_{\chi \leq \xi} \binom{\xi}{\chi}
   \delta^{-|\gamma+\xi-\chi|}
   |D_y^\chi \lbrack(y-x)^\beta\rbrack|,
\end{split}
\end{align}
where we used (\ref{lambdabound1}) to bound $\lambda_j$ and set
\[
\wt{C} := \frac{\|\omega_0\|}{\tau_*}
   \biggl( \sum_{j=1}^M
   \|\Phi\, p_j \|_{W^{m,\infty}(B_1(0))} \biggr).
\]
Since $y \in B_\delta(x)$, we have
\[
|D_y^{\chi}(y-x)^\beta|
\leq C\, \delta^{|\beta-\chi|}.
\]
Plugging this back into~\eqref{E2_F_est} yields~\eqref{claim:F_bg}, finishing the proof.
\end{proof}
\end{lemma}

We are now in a position to compile our results and prove Theorem~\ref{thm:main} as well as Lemma~\ref{lem:multilevel}. To show Theorem~\ref{thm:main}, we combine Lemmas~\ref{lem:E1} and \ref{lem:E2} as well as the classical MLS error estimate from Theorem~\ref{thm:classic_error}.
\bigskip

\noindent\textit{Proof of Theorem~\ref{thm:main}.}
Let $x \in \mathbb{R}^d$ and let $\gamma \in \mathbb{N}_0^d$ satisfy
$r+1 \leq |\gamma|\leq k$. Then
\[
\delta^{|\gamma|} \bigl|D^\gamma E_h f(x) \bigr|
\leq
\delta^{|\gamma|} \bigl|\epsilon^\mathrm{prod}_\gamma(x) \bigr|
+ \delta^{|\gamma|} \bigl|\epsilon^\mathrm{quad}_\gamma(x) \bigr|
+ \delta^{|\gamma|} \bigl|E_h D^\gamma f(x) \bigr|.
\]
Let us treat the two leading terms first. By Lemma~\ref{lem:E1} and Lemma~\ref{lem:E2}, 
\begin{align*}
  \begin{split}
    \delta^{|\gamma|} \bigl|\epsilon^\mathrm{prod}_\gamma(x) \bigr|
+ \delta^{|\gamma|} \bigl|\epsilon^\mathrm{quad}_\gamma(x) \bigr| \leq&
\bigl(C_{1,\gamma} + C_{2,\gamma}\bigr)\, \nu^{|\gamma|-m}\, 
|f|_{W^{r+1,k,\infty}_\delta(B_\delta(x))}
\\ 
&+ 
C_{2,\gamma}\, |f|_{W^{\min\{|\gamma|+r+1,k\},\infty}_\delta(B_\delta(x))}.
\end{split}
\end{align*}
Next, estimating the classical MLS error $E_h D^\gamma f(x)$ via Theorem~\ref{thm:classic_error}, we find
\[
\delta^{|\gamma|} |E_h D^\gamma f(x) |\leq C\delta^{|\gamma|} |D^\gamma f |_{W_\delta^{\min\{r+1,k-|\gamma|\},\infty}(B_\delta(x))} \leq C\, |f |_{W_\delta^{\min\{|\gamma|+r+1,k\},\infty}(B_\delta(x))}.
\]
Taking the maximum over all $|\gamma|= n$ yields the desired estimate. \qed
\bigskip

Next, we give a proof for Lemma~\ref{lem:multilevel} using an
iterative argument based on Theorem~\ref{thm:main}. The reader may
want to recall the definitions of the families $(R_j)_{0\leq j\leq L}$
and $(\theta_j)_{0\leq j\leq L}$, where $R_0 > 0$ is arbitrary, $\theta_0=0$ and 
\[
R_j := R_{j-1} + \delta_{L-j+1} \qquad \text{and} \qquad \theta_{j} :=
\min\{j(r+1), k - r - 1\} 
\]
for $1 \leq j \leq L$, since these frequently occur during the proof. Furthermore, $C$ denotes the constant from Theorem~\ref{thm:main}.
\bigskip

\noindent\textit{Proof of Lemma~\ref{lem:multilevel}.}
We show~\eqref{true_recursion} using an iterative argument based on Theorem~\ref{thm:main}. Setting $\zeta_j := \theta_j + r + 1$, we introduce 
\[
\sigma_j(x) := |\cE_{L-j} f |_{W^{\zeta_j,\infty}_{\delta_{L-j}}(B_{R_j}(x))} \qquad \text{and} \qquad \xi_j(x) := \cC \nu^{k-m} |\cE_{L-j}f |_{W_{\delta_{L-j}}^{r+1,k,\infty}(B_{R_j}(x))}
\]
for $0 \leq j \leq L$ and claim that
\begin{equation}\label{est:recursion}
\sigma_j(x) \leq \mu^{r+1} \xi_{j+1}(x) + \cC\,\mu^{\zeta_{j+1}} \sigma_{j+1}(x),\, \quad 0 \leq j \leq L-1.
\end{equation}
Indeed, as an immediate consequence of estimate~\eqref{claim:main} and the scaling relations between weighted Sobolev norms, we have
\begin{align*} 
\sigma_j(x) &= |\cE_{L-j} f |_{W^{\zeta_j,\infty}_{\delta_{L-j}}(B_{R_j}(x))} = |E_{h_{L-j}} \cE_{L-j-1} f |_{W^{\zeta_j,\infty}_{\delta_{L-j}}(B_{R_j}(x))}  
\\
&\leq
\cC\,\bigl(\nu^{k-m} |\cE_{L-j-1} f|_{W^{r+1,k,\infty}_{\delta_{L-j}}(B_{R_j + \delta_{L-j}}(x))} \\ &\qquad\qquad + |\cE_{L-j-1} f|_{W^{\min\{\zeta_j+r+1,k\},\infty}_{\delta_{L-j}}(B_{R_j+\delta_{L-j}}(x))} \bigr) 
\\
&\leq 
\cC\,\bigl(\nu^{k-m} \mu^{r+1} |\cE_{L-j-1} f|_{W^{r+1,k,\infty}_{\delta_{L-j-1}}(B_{R_{j+1}}(x))} \\ &\qquad\qquad + \mu^{\min\{\zeta_j+r+1,k\}} |\cE_{L-j-1}f|_{W^{\min\{\zeta_j+r+1,k\},\infty}_{\delta_{L-j-1}}(B_{R_{j+1}}(x))}\bigr)
\\ & =
\mu^{r+1} \xi_{j+1}(x) + \cC\,\mu^{\zeta_{j+1}} \sigma_{j+1}(x).
\end{align*}
Iterating~\eqref{est:recursion}, we obtain with $\zeta_0=r+1$,
\begin{equation}\label{recursion_1st_part}
\sigma_0(x) \leq \sum_{j=0}^{L-1} \cC^j \mu^{\zeta_0+\cdots+\zeta_j}
\xi_{j+1}(x) + \cC^L \mu^{\zeta_1+\cdots+\zeta_L} \sigma_L(x). 
\end{equation}
By~\eqref{claim:main_secondary} and using
$\rho=\nu^{k-m}+\mu^{\theta_1}$, it follows for all $1 \leq j \leq L-1$ that
\begin{align*}
\xi_j(x) &= \cC \nu^{k-m} |\cE_{L-j}f |_{W_{\delta_{L-j}}^{r+1,k,\infty}(B_{R_j}(x))}
\\ &\leq \cC^2 \nu^{k-m} \bigl(\nu^{k-m}\mu^{r+1} +
\mu^{\min\{2r+2,k\}}\bigr)
\,|\cE_{L-j-1}f|_{W_{\delta_{L-j-1}}^{r+1,k,\infty}(B_{R_{j+1}}(x))} 
\\ &= \cC\,\mu^{r+1}\rho\,\xi_{j+1}(x)
=\cC^{L-j} \mu^{(L-j)(r+1)} \rho^{L-j} \xi_L(x).
\end{align*}
Inserting this into~\eqref{recursion_1st_part}, one sees that
\begin{align*}
\begin{split}
\sigma_0(x) &\leq  \sum_{j=0}^{L-1} \cC^j
\mu^{\zeta_0+\cdots+\zeta_j}
\cC^{L-j-1}(\mu^{r+1}\rho)^{L-j-1}\xi_{L}(x) + \cC^L
\mu^{\zeta_1+\cdots+\zeta_L} \sigma_L(x) 
\\ &= \cC^{L-1}\biggl(\sum_{j=0}^{L-1} \mu^{\zeta_0+\cdots+\zeta_j}
(\mu^{r+1}\rho)^{L-j-1}\biggr) \xi_{L}(x)+ \cC^L
\mu^{\zeta_1+\cdots+\zeta_L} \sigma_L(x) 
\\ &= \cC^{L-1} \mu^{L(r+1)} \biggl(\sum_{j=0}^{L-1}
\mu^{(\zeta_0-\zeta_0)+\cdots+(\zeta_j-\zeta_0)}
\rho^{L-j-1}\biggr)\xi_L(x) + \cC^L \mu^{L(r+1)}
\mu^{(\zeta_1-\zeta_0)+\cdots+(\zeta_L-\zeta_0)} \sigma_L(x). 
\end{split}
\end{align*}
Since $\zeta_j - \zeta_0 = \theta_j$, we may substitute $\sigma_0(x)$, $\xi_L(x)$ by their definitions and bound $\sigma_L(x)$ from above via
\[
\sigma_L(x) = |f |_{W^{\zeta_L,\infty}_{\delta_{0}}(B_{R_L}(x))} \leq |f |_{W^{r+1,k,\infty}_{\delta_{0}}(B_{R_L}(x))}
\]
to finally obtain~\eqref{true_recursion}. For~\eqref{true_recursion_simplified}, simply observe that
\[
\cC \nu^{k-m} \sigma_0(x) \leq \xi_0(x) \leq \cC^L \mu^{L(r+1)} \rho^L \xi_L(x) = \cC^L \mu^{L(r+1)} \rho^L \cdot \cC\nu^{k-m} |f|_{W_{\delta_0}^{r+1,k,\infty}(B_{R_L}(x))}.
\]
Dividing both sides by $\cC \nu^{k-m}$ yields the claim. \qed
\bigskip
\noindent

Before we conclude this section, we still need to give a proof for
Lemma~\ref{lem:convergence_gain}, in which we studied the convergence
gaining term $J_L(\nu,\mu)$ from~\eqref{def:J_L} more closely. Here,
$C>0$ denotes the constant from Theorem~\ref{thm:main} again. 
\bigskip

\noindent\textit{Proof of Lemma~\ref{lem:convergence_gain}.}
By definition, $\rho = (1+\kappa)\mu^{\theta_1}$ and $\theta_1 =
r+1$. Substituting this, leads to
\[
S := \nu^{k-m}\sum_{j=0}^{L-2}\mu^{\theta_0+\dots+\theta_j} \rho^{L-j-2} = \kappa(1+\kappa)^{L-2}\mu^{(L-1)\theta_1} \sum_{j=0}^{L-2} T_j,
\]
where $T_j := \mu^{\theta_0+\dots+\theta_j-j\theta_1}(1+\kappa)^{-j}$. Since $k \geq L(r+1)$, we have $T_j = \mu^{\frac{1}{2}j(j-1)(r+1)}(1+\kappa)^{-j}$. Shifting the summation index by one and using the inequality $\frac{1}{2}(r+1)(j+1)j \geq j$, we can bound the sum via
\[
\sum_{j=0}^{L-2} T_j \leq 1 + \frac{1}{1+\kappa} \sum_{j=0}^\infty\mu^{j}(1+\kappa)^{-j} = 1 + \frac{1}{1-\mu+\kappa} \leq \frac{2-\mu}{1-\mu}.
\]
Substituting this bound back into our expression for $S$ gives
\[
S \leq \left(\frac{\kappa}{1+\kappa}\right)\left(\frac{2-\mu}{1-\mu}\right) \mu^{(L-1)(r+1)}(1+\kappa)^{L-1}.
\]
The claim then follows directly from this estimate, after factoring out $\cC\rho = \cC (1+\kappa) \mu^{r+1}$ from
\[
J_{L-1}(\nu,\mu) = \min\{\cC\rho,\, \cC(\mu^{\frac{1}{2}L(L-1)(r+1)} + S)^{1\slash(L-1)}\}.
\]
\qed

\section{Numerical Experiments}\label{sec:numerik}

In this section, we present some numerical experiments that validate
our theory. The convergence rate of the classical MLS method will be
compared to the convergence rate of the multilevel method. More
precisely, for every test run, a tuple $(h_0, \nu, \mu, r, L, \Phi)$
of an initial mesh width $h_0 >0$, a  ratio $\nu > 1$ between support
size and mesh width, a refinement factor $0 < \mu < 1$, a polynomial reproduction degree $r\in\mathbb{N}_0$, a number of levels $L \in \mathbb{N}$ and a kernel $\Phi : \R^d \to \R$ is chosen. 

Given a function $f: \R^d \to \R$ to be approximated, we then calculate the discrete $\ell^\infty$ norm of both the error sequences $\|E_{h_\ell} f \|_{\ell^\infty(Y_L)}$ and $\|\cE_\ell f \|_{\ell^\infty(Y_L)}$ over the set
\[
Y_L \defg \{y \in {\intervalc 0 1}^d \,:\, (4 y) \slash h_L \in \mathbb{Z}^d\},
\]
i.e. we take the maximum over all points in the cube ${\intervalc 0 1}^d$ which lie on a grid four times finer than $X_{h_L}$. The errors are plotted logarithmically, i.e. the graphs will show the map
\[
\mathfrak{E}_{SL}:\{1,\ldots,L\} \ni \ell \mapsto \log_{10}\bigl(\|E_{h_\ell} f\|_{\ell^\infty(Y_L)}\bigr) 
\]
depicting the error of the standard MLS method, as well as the multilevel error
\[
\mathfrak{E}_{ML} : \{1,\ldots,L\} \ni \ell \mapsto \log_{10}\bigl(\|\cE_\ell f \|_{\ell^\infty(Y_L)}\bigr).
\]
To calculate numerical convergence rates for each method, we calculate the two vectors
\[
\mathfrak{r}_{SL} :=
\biggl\lbrack\log_{\mu}\biggl(\frac{\|E_{h_{\ell+1}} f
  \|_{\ell^\infty(Y_L)}}{\|E_{h_\ell} f
  \|_{\ell^\infty(Y_L)}}\biggr)\biggr\rbrack_{\ell=1,\ldots,L-1} \quad
\text{and} \quad \mathfrak{r}_{ML} :=
\biggl\lbrack\log_{\mu}\biggl(\frac{\|\cE_{\ell+1} f
  \|_{\ell^\infty(Y_L)}}{\|\cE_\ell f
  \|_{\ell^\infty(Y_L)}}\biggr)\biggr\rbrack_{\ell=1,\ldots,L-1}, 
\]
which are then displayed in tables for reading. In the figures below,
$\mathfrak{E}_{SL}$ will be depicted in blue and
$\mathfrak{E}_{SL}$  in orange. For the classical MLS
method, we will always expect a convergence rate of $r+1$ if  $r$ is
odd, and a rate of $r+2$ if $r$ is even,
c.f. Theorem~\ref{thm:classic_error} and
Theorem~\ref{thm:classic_error_improved}. 

In our examples, the kernels will be chosen as Wendland functions
$\Phi_{d,k} = \phi_{d,k}(\|\cdot\|) \in C^{2k}(\R^d)$, where the
family of functions $(\phi_{d,k})_{d,k\in\N}$ is defined as in Chapter
9 of~\cite{Wendland-05-1}. We will only deal with two-dimensional
problems, i.e. $d=2$. As for the initial mesh and the refinement rate,
we will stick to $h_0 = 0.25$ and $\mu = 0.5$. The content of all
Figures was generated using the python library
matplotlib~\cite{Hunter-07-1}. The MLS algorithm and its multilevel
variant were implemented in C++. The code is made available in
repository~\cite{Durst-26-7} for reproducibility. 

\begin{example}\label{example_1}
In the remarks following Theorem~\ref{thm:result}, we noted that the
multilevel scheme can achieve nearly twice the convergence rate of the
standard MLS method if $k=2r+2$, provided the abstract constants in
our estimates are sufficiently small. Here, we explore this scenario
with the specific choice $r=1$, yielding an expected convergence rate
of $2(r+1) = 4$. Consider the function $f:\R^2 \to \R$ defined by 
\[
f(x) := \|x - (0.5,0.5)^\transpose \|^{4+\eps}\,,
\]
where $\eps = 0.01$. This function satisfies $f \in
C^{4,\eps}(\R^2)\setminus C^5(\R^2)$, and we approximate it using the
Wendland kernel $\Phi_{2,3} \in C^6(\R^2)$. The results are visualized
in Figure~\ref{fig:example_1}. As shown in
Table~\ref{tab:example_1_combined}, the multilevel scheme indeed
attains convergence rates of approximately $4$. Furthermore, a
significant speedup is observed at $L=3$ for $\nu=8.1$. 
\end{example}

\begin{example}\label{example_2}
While Example~\ref{example_1} addressed the case $k=2r+2$, our estimates suggest that the multilevel scheme attains even higher convergence orders if $k \geq L(r+1)$ (cf.\ Lemma~\ref{lem:convergence_gain}). These supplementary gains diminish as $L \to \infty$. To illustrate the efficacy of the multilevel method under optimal conditions, we approximate the smooth function $f:\R^2 \to \R$ given by
\[
f(x) = \sin(\|x\|^2)(1+\cos(\|x\|^2))\,,
\]
using the kernel $\Phi_{2,3}$ and a polynomial reproduction degree of
$r=0$. Because $r$ is even, Theorem~\ref{thm:classic_error_improved}
predicts a classical convergence rate of $r+2=2$. Naturally, this
super-convergence of the standard MLS scheme also enhances the overall
convergence of the multilevel method. The results are depicted in
Figure~\ref{fig:example_2} and
Table~\ref{tab:example_2_combined}. Following a substantial initial
speedup, the convergence rate begins to stagnate around $L=6$ for
$\nu=8.1$. 
\end{example}

\begin{figure}[H]
    \centering
    \begin{minipage}[b]{0.47\textwidth}
        \centering
        \includegraphics[width=\textwidth]{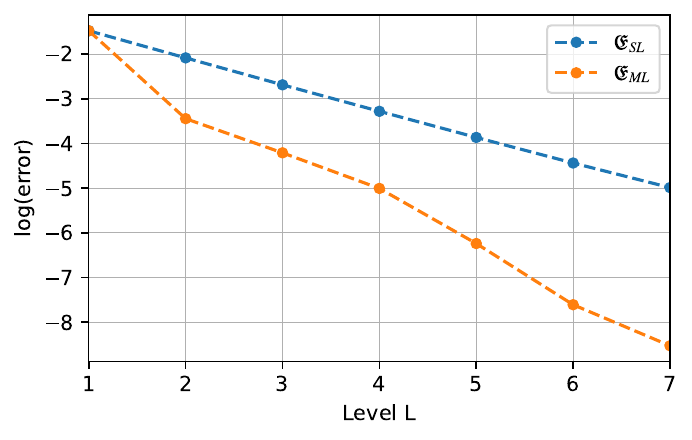}
        \label{fig:ex1_plot_a}
    \end{minipage}%
    \hfill%
    \begin{minipage}[b]{0.48\textwidth}
        \centering
        \includegraphics[width=\textwidth]{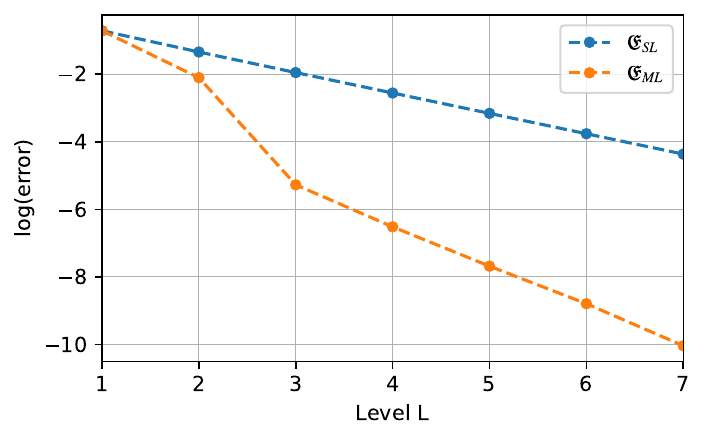}
        \label{fig:ex1_plot_b}
    \end{minipage}
    \caption{Errors $\mathfrak{E}_{SL}(\ell)$ and $\mathfrak{E}_{ML}(\ell)$ for Example~\ref{example_1} and levels $1 \leq \ell \leq L = 7$.}
    \label{fig:example_1}
\end{figure}

\vspace{0.5cm}

\begin{table}[H]
\centering
\begin{tabular}{cl @{:\quad} cccccc}
\toprule
\multicolumn{2}{c}{} & $\ell=1$ & $\ell=2$ & $\ell=3$ & $\ell=4$ & $\ell=5$ & $\ell=6$ \\
\midrule
$\nu = 3.5$ & $\mathfrak{r}_{SL}(\ell)$ & 2.0168 & 2.0042 & 2.0011  & 2.0003 & 2.0001 & 1.9981 \\
            & $\mathfrak{r}_{ML}(\ell)$ & 6.5699 & 3.2004 & 2.7497 & 4.8132 & 4.1463 & 3.0289 \\
\midrule
$\nu = 8.1$ & $\mathfrak{r}_{SL}(\ell)$ & 2.0861 & 2.0224 & 2.0056 & 2.0014 & 2.0004 & 2.0001 \\
            & $\mathfrak{r}_{ML}(\ell)$ & 4.6101 & 10.5224 & 4.1322 & 3.8686 & 3.6750 & 4.0164 \\
\bottomrule
\end{tabular}
\caption{Rates $\mathfrak{r}_{SL}$ and $\mathfrak{r}_{ML}$ for Example~\ref{example_1} under varying $\nu$ parameters.}
\label{tab:example_1_combined}
\end{table}

\vspace{2.0cm}

\begin{figure}[H]
    \centering
    \begin{minipage}[b]{0.47\textwidth}
        \centering
        \includegraphics[width=\textwidth]{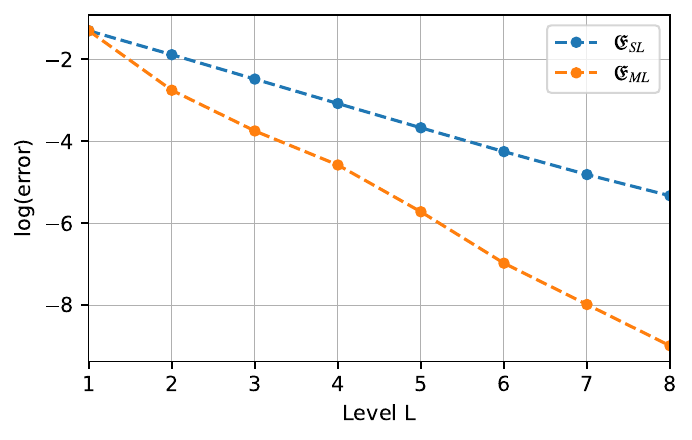}
        \label{fig:ex2_plot_a}
    \end{minipage}%
    \hfill%
    \begin{minipage}[b]{0.48\textwidth}
        \centering
        \includegraphics[width=\textwidth]{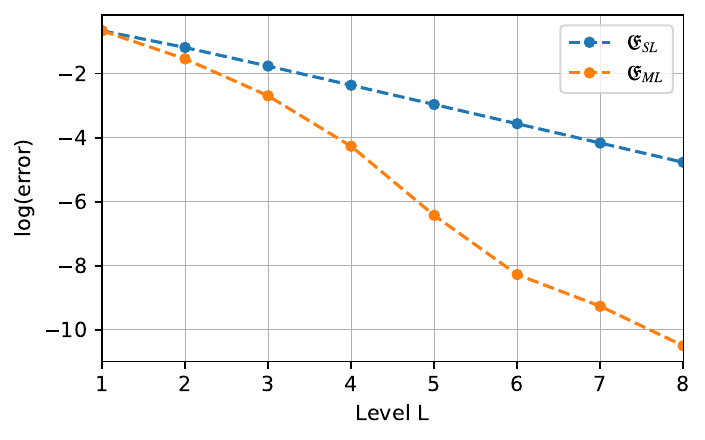}
        \label{fig:ex2_plot_b}
    \end{minipage}
    \caption{Errors $\mathfrak{E}_{SL}(\ell)$ and $\mathfrak{E}_{ML}(\ell)$ for Example~\ref{example_2} and levels $1 \leq \ell \leq L = 8$.}
    \label{fig:example_2}
\end{figure}

\vspace{0.5cm}

\begin{table}[H]
\centering
\begin{tabular}{cl @{:\quad} ccccccc}
\toprule
\multicolumn{2}{c}{} & $\ell=1$ &  $\ell=2$ & $\ell=3$ & $\ell=4$ & $\ell=5$ & $\ell=6$ & $\ell=7$ \\
\midrule
$\nu = 3.5$ & $\mathfrak{r}_{SL}(\ell)$ & 1.9425 & 1.9791 & 1.9814 & 1.9678 & 1.9342 & 1.8618 & 1.7178 \\
            & $\mathfrak{r}_{ML}(\ell)$ & 4.8245 & 3.3032 & 2.7603 & 3.7968 & 4.1690 & 3.3528 & 3.3437 \\
\midrule
$\nu = 8.1$ & $\mathfrak{r}_{SL}(\ell)$ & 1.7356 & 1.9285 & 1.9818 & 1.9954 & 1.9988 & 1.9997 & 1.9999 \\
            & $\mathfrak{r}_{ML}(\ell)$ & 2.9146 & 3.8334 & 5.2288 & 7.1642 & 6.1393 & 3.2821 & 4.0965 \\
\bottomrule
\end{tabular}
\caption{Rates $\mathfrak{r}_{SL}$ and $\mathfrak{r}_{ML}$ for Example~\ref{example_2} under varying $\nu$ parameters.}
\label{tab:example_2_combined}
\end{table}


\end{document}